\documentclass[11pt]{amsart}

\usepackage{amsfonts,epsfig}
\usepackage{latexsym}
\usepackage{amssymb}
\usepackage{amsmath}
\usepackage{amsthm}
\usepackage{graphics}
\usepackage[all]{xy}
\usepackage[T2A]{fontenc}
\usepackage{multirow}
\usepackage{hhline}
\usepackage{array, booktabs, ctable}
\usepackage[backref]{hyperref}
\usepackage{enumerate}
\usepackage{comment}
\usepackage{breqn}
\usepackage{mathtools}
\usepackage{verbatim}

\newtheorem{defn}{Definition}[section]
\newtheorem{definition}[defn]{Definition}

\newtheorem{lemma}[defn]{Lemma}

\newtheorem{theorem}[defn]{Theorem}

\newtheorem{proposition}[defn]{Proposition}

\newtheorem*{theorem*}{Theorem}

\theoremstyle{definition}

\newtheorem*{ack}{Acknowledgements}
\newtheorem{remark}[defn]{Remark}

\newtheorem{example}[defn]{Example}

\newcommand{\Q}{\mathbb Q}
\newcommand{\Z}{\mathbb Z}

\newcommand{\Gal}{\operatorname{Gal}}
\newcommand{\Aut}{\operatorname{Aut}}

\newcommand{\GL}{\operatorname{GL}}

\begin{document}



%\bibliographystyle{plain}

\title[Maximal Abelian Extenstion Contained in a Division Field and CM]{The maximal abelian extension contained in a division field of an elliptic curve over $\mathbb{Q}$ with complex multiplication}

\author{Asimina S. Hamakiotes}
\address{University of Connecticut, Department of Mathematics, Storrs, CT 06269, USA}
\email{asimina.hamakiotes@uconn.edu} 
\urladdr{https://asiminah.github.io/}

\newcommand{\asimina}[1]{{\color{blue} \sf Asimina: [#1]}}

\begin{abstract} Let $K$ be an imaginary quadratic field, and let $\mathcal{O}_{K,f}$ be an order in $K$ of conductor $f\geq 1$. Let $E$ be an elliptic curve with CM by $\mathcal{O}_{K,f}$, such that $E$ is defined by a model over $\mathbb{Q}(j_{K,f})$, where $j_{K,f}=j(E)$. It has been shown by the author and Lozano-Robledo that $\Gal(\Q(j_{K,f},E[N])/\Q(j_{K,f}))$ is only abelian for $N=2,3$, and $4$. Let $p$ be a prime and let $n\geq 1$ be an integer. In this article, we bound the commutator subgroups of $\Gal(\Q(E[p^n])/\Q)$ and classify the maximal abelian extensions contained in $\Q(E[p^n])/\Q$. 
\end{abstract}

\maketitle

%%%%%%%%%%%%%%%%%%%%%%%%%%%%%%%%%%%%%%%%%%%%%%%%%%%%%%%%%%%%%%%%%%%%%%%%%%%%%%%%

\section{Introduction}

Let $F$ be a number field and let $E$ be an elliptic curve defined over $F$. Let $\overline{F}$ be a fixed algebraic closure of $F$. Let $N\geq 2$ be an integer and let $E[N] = E(\overline{F})[N]$ be the $N$-torsion subgroup of $E(\overline{F})$. Let $F(E[N])$ be the $N$-division field of $E/F$, i.e., the field of definition of the coordinates of points in $E[N]$. It is natural to study the extension $F(E[N])/F$ and since it is Galois, it is also natural to study its Galois group $\Gal(F(E[N])/F)$, and the Galois representations that arise from the natural action of Galois on the $N$-torsion subgroup of the elliptic curve $E$.

The existence of the Weil pairing implies that $F(E[N])$ contains all the $N$-th roots of unity of $\overline{F}$, i.e., $F(\zeta_N) \subseteq F(E[N])$, where $\zeta_N$ is a primitive $N$-th root of unity \cite[Ch. III, Cor. 8.1.1]{silverman1}. This raises the following two questions:
\begin{itemize}
    \item[(1)] When is the extension $F(E[N])/F$ abelian? %, i.e., when is $\Gal(F(E[N])/F)$ is abelian?  
    \item[(2)] What is the maximal abelian extension contained in $F(E[N])/F$?
\end{itemize}

When $F=\Q$, Gonz\'alez-Jim\'enez and Lozano-Robledo \cite{gonzalez-jimenez-lozano-robledo} answer (1) and prove that if $\Q(E[N])/\Q$ is abelian, then $N=2, 3, 4, 5, 6,$ or $8$, and classify all the elliptic curves $E/\Q$ with this property. Moreover, for each $N$, they classify the possible abelian Galois groups that occur for a division field. When $F\neq \Q$, the question of when $F(E[N])/F$ is abelian is generally harder to answer; however, it has been answered for elliptic curves $E/F$ with complex multiplication (CM), where $F$ is the minimal field of definition of $E$, i.e., $F=\Q(j(E))$ where $j(E)$ is the $j$-invariant of $E$. In \cite{hamakiotes2023elliptic}, the author and Lozano-Robledo answer (1) by proving that for an elliptic curve $E/F$ with CM, where $F = \Q(j(E))$, the extension $F(E[N])/F$ is abelian only for $N=2,3,$ or $4$. Furthermore, they classify all elliptic curves $E/F$ with this property and for each $N$, they classify the possible abelian Galois groups that occur for a division field. Since the extension $F(E[N])/F$ is usually not abelian, it is natural to wonder what the maximal abelian extension contained in $F(E[N])/F$ is. 

The goal of this article is to answer (2) by providing a classification of the maximal abelian extensions contained in $\Q(E[N])/\Q$ for elliptic curves $E/\Q$ with CM. Let $K$ be an imaginary quadratic field and let $\mathcal{O}_K$ be the ring of integers of $K$ with discriminant $\Delta_K$. Let $f\geq 1$ be an integer and let $\mathcal{O}_{K,f}$ be the order of $K$ of conductor $f$, with discriminant $\Delta_Kf^2$. Let $j_{K,f}$ be a $j$-invariant associated to the order $\mathcal{O}_{K,f}$, i.e., an arbitrary Galois conjugate of $j(\mathbb{C}/\mathcal{O}_{K,f})$. Let $E/\Q$ be an elliptic curve with CM by $\mathcal{O}_{K,f}$, so $j_{K,f}\in \Q$, and let $G_{E,N} = \Gal(\Q(E[N])/\Q)$ as defined in Section \ref{notation}. By Theorem \ref{gal-repns-att-to-ec-w-cm} of Section \ref{classification-of-images}, the Galois group $G_{E,N}$ is contained in a subgroup of $\GL(2,\Z/N\Z)$ called the normalizer of Cartan $\mathcal{N}_{\delta,\phi}(N)$, which we will define in Section \ref{notation}. The index of $G_{E,N}$ in $\mathcal{N}_{\delta,\phi}(N)$ is a divisor of $4$ if $j(E)=1728$, a divisor of $6$ if $j(E)=0$, and a divisor of $2$ otherwise. 
Let $M_E(N)$ denote the maximal abelian extension contained in $\Q(E[N])/\Q$, i.e., $M_E(N) = \Q(E[N])\cap \Q^{\text{ab}}$.

The main results of this paper are as follows. Suppose that $p$ is an odd prime not dividing $\Delta_Kf$. 
\begin{theorem}\label{pndivDeltaKf}
    Let $E/\Q$ be an elliptic curve with CM by an order $\mathcal{O}_{K,f}$ of $K$ of conductor $f\geq 1$. Let $p$ be an odd prime not dividing $\Delta_Kf^2$ and let $n\geq 1$. Then $M_E(p^n) = K(\zeta_{p^n})$.
\end{theorem}

Suppose that $p$ is an odd prime dividing $\Delta_Kf$. Let $E/\Q$ be an elliptic curve with CM by an order $\mathcal{O}_{K,f}$ of $K$ of conductor $f\geq 1$. If $j(E)\neq 0, 1728$, then $E$ is a quadratic twist of $E_{\Delta_K,f}$ from Table \ref{ell_curves} by some $d \in \Z$ that is unique up to squares and that determines $\alpha = \alpha(E)$, which we will discuss in Section \ref{computing-alpha}. 

\begin{theorem}\label{pdivDeltaKf}
    Let $E/\Q$ be an elliptic curve with CM by an order $\mathcal{O}_{K,f}$, $f\geq 1$, such that $j(E)\neq 0$ or $1728$. Let $p$ be an odd prime dividing $\Delta_Kf^2$ and let $n\geq 1$. Then there is an explicitly computable integer $\alpha = \alpha(E)$, unique up to squares, such that $M_E(p^n) = \Q(\zeta_{p^n},\sqrt{\alpha})$. Moreover, if $[\mathcal{N}_{\delta,\phi}(p^n): G_{E,p^n}] = 2$, then $\sqrt{\alpha} \in \Q(\zeta_{p^n})$ and $M_E(p^n) = \Q(\zeta_{p^n})$.
\end{theorem}

Let $p=3$ and $\Delta_Kf^2 = -3$. If $j(E)=0$, then $E$ is a sextic twist of $E_{-3,1}$ by some $d \in \Z$ that is unique up to $6$-th powers and that determines $\alpha = \alpha(E)$, which we will discuss in Section \ref{computing-alpha}. 
\begin{theorem}\label{3divDeltaKf}
    Let $E/\Q$ be an elliptic curve with $j(E)=0$ and let $n\geq 1$. Then there is an explicitly computable integer $\alpha = \alpha(E)$, unique up to $6$-th powers, such that $M_E(3^n) = \Q(\zeta_{3^n},\sqrt{\alpha})$. Moreover, if $[\mathcal{N}_{-1,1}(3^n): G_{E,3^n}] = 2$ or $6$, then $\sqrt{\alpha} \in \Q(\zeta_{3^n})$ and $M_E(3^n) = \Q(\zeta_{3^n})$.
\end{theorem}

The cases where $p=2$ are handled separately in the following theorem. 
Note that if $j(E)=1728$, then $E$ is a quartic twist of $E_{-4,1}$ by some $d \in \Z$ that is unique up to $4$-th powers and that determines $\alpha = \alpha(E)$, which we will discuss in Section \ref{computing-alpha}.

\begin{theorem}\label{casep2}
    Let $E/\Q$ be an elliptic curve with CM by an order $\mathcal{O}_{K,f}$ of $K$ of conductor $f \geq 1$. %Let $n\geq 1$ and let $G_{E,2^n} = \Gal(\Q(E[2^n])/\Q)$. 
    \begin{enumerate}
        \item[(a)] Suppose that $2$ does not divide $\Delta_Kf$. Then $M_E(2^n) = K(\zeta_{2^n})$ for $n\geq 1$.
        
        \item [(b)] Suppose that $2$ divides $\Delta_Kf$. There is an explicitly computable integer $\alpha = \alpha(E)$, unique up to squares, such that:  
        \begin{enumerate} 
            \item[(i)] If $\Delta_Kf^2 = -12$ or $-28$, then $M_E(2^n) = K(\zeta_{2^{n+1}})$ for $n \geq 2$. For $n=1$, 
            \begin{itemize}
                \item if $\Delta_Kf^2 = -12$, then $M_E(2) = \Q(\sqrt{3})$; and, 
                \item if $\Delta_Kf^2 = -28$, then $M_E(2) = \Q(\sqrt{7})$.
            \end{itemize}
            \item[(ii)] If $\Delta_Kf^2 = -4$, then $M_E(2^n) = \Q(\zeta_{2^{n+1}},\sqrt{\alpha})$ for $n \geq 2$. Moreover, if $[\mathcal{N}_{\delta,\phi}(2^n):G_{E,2^n}] = 4$, then $\sqrt{\alpha} \in \Q(\zeta_{2^{n+1}})$ and $M_E(2^n) = \Q(\zeta_{2^{n+1}})$. For $n=1$, 
            \begin{itemize}
                \item if $E : y^2 = x^3 + t^2x$ for $t\in \Q^*$, then $M_E(2) = \Q(i)$; 
                \item if $E : y^2 = x^3 - t^2x$ for $t \in \Q^*$, then $M_E(2) = \Q$;  
                \item if $E : y^2 = x^3 + sx$ where $s \neq \pm t^2$ for $t\in \Q^*$, then $M_E(2) = \Q(\sqrt{-s})$.
            \end{itemize}

            \item[(iii)] If $\Delta_Kf^2 = -8$ or $-16$, then $M_E(2) = \Q(\sqrt{2})$ and $M_E(2^n) = \Q(\zeta_{2^{n+1}},\sqrt{\alpha})$ for $n\geq 2$. Moreover, if $[\mathcal{N}_{\delta,\phi}(2^n): G_{E,2^n}] = 2$, then $\sqrt{\alpha} \in \Q(\zeta_{2^{n+1}})$ and $M_E(2^n) = \Q(\zeta_{2^{n+1}})$.
        \end{enumerate}
    \end{enumerate}
\end{theorem}

\subsection{Structure of the paper} The paper is organized as follows. In Section \ref{examples} we provide examples of elliptic curves that satisfy the conditions of Theorems \ref{pndivDeltaKf}, \ref{pdivDeltaKf}, \ref{3divDeltaKf}, and \ref{casep2}. In Section \ref{background}, we establish notation, recall some results about known abelian extensions contained in division fields, state preliminary results on $\alpha$ in division fields, recall important results about the classification of $\ell$-adic images, and state some preliminary results about commutator subgroups. In Section \ref{main_method}, we outline the main method used throughout the proofs of the main results. In Section \ref{proofs_main_results}, we prove the main results: Theorem \ref{pndivDeltaKf} is shown in Section \ref{sect-pndivdisc}, Theorems \ref{pdivDeltaKf} and \ref{3divDeltaKf} are shown in Section \ref{sect-pdivdisc}, and Theorem \ref{casep2} is shown in Section \ref{sect-casep2}.

\subsection{Code} The \verb|Magma| \cite{Magma} code verifying the computational claims made in the paper can be found at GitHub repository \cite{GitHubPaperCode},
\begin{center}
    \url{https://github.com/asiminah/max-ab-extn-contained-in-div-flds}, 
\end{center}
which also includes code to compute $\alpha$, which is described in Definition \ref{defn-alpha}, and code to compute $M_E(p^n)$ for $E/\Q$ with CM and $p$ prime.

\begin{ack} The author would like to thank \'Alvaro Lozano-Robledo for helpful comments and suggestions, and Benjamin York for clarifying conversations.
\end{ack}

%%%%%%%%%%%%%%%%%%%%%%%%%%%%%%%%%%%%%%%%%%%%%%%%%%%%%%%%%%%%%%%%%%%%%%%%%%%%%%%%

\section{Examples}\label{examples}

In this section, we provide examples of elliptic curves (with LMFDB labels, \cite{lmfdb}) that satisfy the conditions of Theorems \ref{pndivDeltaKf}, \ref{pdivDeltaKf}, \ref{3divDeltaKf}, and \ref{casep2}. All examples can be verified using the code in \cite{GitHubPaperCode}.

\subsection{}
The following is an example of Theorem \ref{pndivDeltaKf}. 
\begin{example}
    Let $E/\Q: y^2=x^3-35x+98$ (\href{https://www.lmfdb.org/EllipticCurve/Q/784/f/4}{784.f1}), where $j(E) = -3375$. Here $K=\Q(\sqrt{-7})$. In this case, $G_{E,5^n} = \mathcal{N}_{\delta,0}(5^n)$, where $\delta = -7/4$. Thus, $M_E(5^n) = K(\zeta_{5^n})$.
\end{example}

\subsection{}
The following is an example of Theorem \ref{pdivDeltaKf}. 
\begin{example}
    Let $E/\Q: y^2=x^3-140x-784$ (\href{https://www.lmfdb.org/EllipticCurve/Q/3136/n/4}{3136.n4}), where $j(E) = -3375$. Here $K=\Q(\sqrt{-7})$. In this case, $G_{E,7^n} = \mathcal{N}_{\delta,0}(7^n)$, where $\delta = -7/4$. Note that $E$ is a quadratic twist of $E_{-7,1}/\Q: y^2 = x^3 - 1715x + 33614$ (\href{https://www.lmfdb.org/EllipticCurve/Q/49/a/2}{49.a2}) by $-14$. Thus, $M_E(7^n) = \Q(\zeta_{7^n},\sqrt{-14})$. The simplest CM curve $E_{-7,1}$ has image 
    \[
    G_{E_{-7,1},7^n} = \left\langle \begin{pmatrix}-1&0\\0&1\end{pmatrix}, \left\{ \begin{pmatrix}a&b\\\delta b&a\end{pmatrix}: a \in {(\Z/7\Z)^\times}^2, b \in \Z/7\Z \right\} \right\rangle \subseteq \GL(2,\Z/7^n\Z),
    \]
    which is an index 2 subgroup of $\mathcal{N}_{\delta,0}(7^n)$. Thus, $M_{E_{-7,1}}(7^n) = \Q(\zeta_{7^n})$.
\end{example}

\subsection{}
The following is an example of Theorem \ref{3divDeltaKf}. 
\begin{example}
	Let $E/\mathbb{Q}: y^2 = x^3 - 6$ (\href{https://www.lmfdb.org/EllipticCurve/Q/15552/z/1}{15552.z1}), where $j(E)=0$. Here $K=\Q(\sqrt{-3})\subseteq \Q(\zeta_{3^n})$. The image
	\[
		G_{E,3^n} = \left\langle \begin{pmatrix}-1&0\\0&1\end{pmatrix}, \begin{pmatrix}2&0\\0&2\end{pmatrix}, \begin{pmatrix}1&1\\-3/4&1\end{pmatrix} \right\rangle \subseteq \operatorname{GL}(2,\mathbb{Z}/3^n\mathbb{Z})
	\]
	is an index 3 subgroup of $\mathcal{N}_{\delta,0}(3^n)$. Note that $E$ is a sextic twist of $E_{-3,1}/\mathbb{Q} : y^2 = x^3 + 16$ (\href{https://www.lmfdb.org/EllipticCurve/Q/27/a/4}{27.a4}) by $-24$. Thus, $M_E(3^n) = \mathbb{Q}(\zeta_{3^n},\sqrt{-24}) = \mathbb{Q}(\zeta_{3^n},\sqrt{2})$. The simplest CM curve $E_{-3,1}$ has image 
	\[
		G_{E_{-3,1},3^n} = \left\langle \begin{pmatrix}-1&0\\0&1\end{pmatrix}, \left\{\begin{pmatrix}a&b\\-3b/4&a\end{pmatrix} : a \equiv 1, b\equiv 0 \bmod 3 \right\} \right\rangle \subseteq \operatorname{GL}(2,\mathbb{Z}/3^n\mathbb{Z}),
	\]
	which is an index 6 subgroup of $\mathcal{N}_{\delta,0}(3^n)$. Thus, $M_{E_{-3,1}}(3^n) = \mathbb{Q}(\zeta_{3^n})$.
\end{example}

\subsection{}
The following are examples of Theorem \ref{casep2}, part (a), (b)(i), (b)(ii), and (b)(iii), respectively.
\begin{example}
    Let $E/\Q: y^2+y=x^3-x^2-7x+10$ (\href{https://www.lmfdb.org/EllipticCurve/Q/121/b/2}{121.b2}), where $j(E) = -32768$. Here $K=\Q(\sqrt{-11})$. In this case, $G_{E,2^n} = \mathcal{N}_{\delta,0}(2^n)$, where $\delta = -11/4$. Thus, $M_E(2^n) = K(\zeta_{2^n})$.
\end{example}

\begin{example}
    Let $E/\Q: y^2=x^3-15x+22$ (\href{https://www.lmfdb.org/EllipticCurve/Q/36/a/2}{36.a2}), where $j(E) = 54000$. Here $K = \Q(\sqrt{-12})$. In this case, $G_{E,2^n} = \mathcal{N}_{\delta, 0}(2^n)$, where $\delta = -3$. Thus, $M_E(2^n) = K(\zeta_{2^{n+1}})$.
\end{example}

\begin{example}
		Let $E/\mathbb{Q}: y^2 = x^3 + 9x$ (\href{https://www.lmfdb.org/EllipticCurve/Q/576/c/4}{576.c4}), where $j(E)=1728$. Here $K=\Q(\sqrt{-4}) \subseteq \Q(\zeta_{2^{n+1}})$. The image
	\[
		G_{E,2^n} = \left\langle \begin{pmatrix}0&1\\1&0\end{pmatrix}, \begin{pmatrix}-1&0\\0&-1\end{pmatrix}, \begin{pmatrix}3&0\\0&3\end{pmatrix}, \begin{pmatrix}1&2\\-2&1\end{pmatrix} \right\rangle \subseteq \operatorname{GL}(2,\mathbb{Z}/2^n\mathbb{Z})
	\]
	is an index 2 subgroup of $\mathcal{N}_{-1,0}(2^n)$. Note that $E$ is a quadratic twist of $E_{-4,1}/\mathbb{Q} : y^2 = x^3 + x$ (\href{https://www.lmfdb.org/EllipticCurve/Q/64/a/4}{64.a4}) by 3. Thus, $M_E(2^n) = \mathbb{Q}(\zeta_{2^{n+1}},\sqrt{3})$. The simplest CM curve $E_{-4,1}$ has image
		\[
			G_{E_{-4,1},2^n} = \left\langle \begin{pmatrix}0&-1\\-1&0\end{pmatrix}, \begin{pmatrix}5&0\\0&5\end{pmatrix}, \begin{pmatrix}1&2\\-2&1\end{pmatrix}\right\rangle \subseteq \operatorname{GL}(2,\mathbb{Z}/2^n\mathbb{Z}),
		\]
		which is an index 4 subgroup of $\mathcal{N}_{-1,0}(2^n)$. Thus, $M_{E_{-4,1}}(2^n) = \mathbb{Q}(\zeta_{2^{n+1}})$.
	\end{example}

\begin{example}
    Let $E/\Q: y^2=x^3-30x+56$ (\href{https://www.lmfdb.org/EllipticCurve/Q/2304/h/2}{2304.h2}), where $j(E) = 8000$. Here $K = \Q(\sqrt{-8})$. In this case, $G_{E,2^n} = \mathcal{N}_{\delta,0}(2^n)$, where $\delta = -2$. Note that $E$ is a quadratic twist of $E_{-8,1}/\Q: y^2 = x^3 - 4320x + 96768$ (\href{https://www.lmfdb.org/EllipticCurve/Q/256/a/2}{256.a2}) by $12$. Thus, $M_E(2^n) = \Q(\zeta_{2^{n+1}}, \sqrt{12}) = \Q(\zeta_{2^{n+1}}, \sqrt{3})$. The simplest CM curve $E_{-8,1}$ has image 
    \[
	G_{E_{-8,1},2^n} = \left\langle \begin{pmatrix}-1&0\\0&1\end{pmatrix},  \begin{pmatrix}3&0\\0&3\end{pmatrix}, \begin{pmatrix}-1&-1\\2&-1\end{pmatrix} \right\rangle \subseteq \operatorname{GL}(2,\mathbb{Z}/2^n\mathbb{Z}),
    \]
    which is an index 2 subgroup of $\mathcal{N}_{\delta,0}(2^n)$. Thus, $M_{E_{-8,1}}(2^n) = \Q(\zeta_{2^{n+1}})$.
\end{example}

%%%%%%%%%%%%%%%%%%%%%%%%%%%%%%%%%%%%%%%%%%%%%%%%%%%%%%%%%%%%%%%%%%%%%%%%%%%%%%%%

\section{Background and Preliminary Results}\label{background}

In this section, we recall some results about abelian extensions contained in the $N$-division field of an elliptic curve with CM, discuss and state some preliminary results on computing $\alpha$, recall important results about the classification of $\ell$-adic images of Galois representations attached to elliptic curves with CM, and state some preliminary results relating to commutator subgroups. 

%%%%%%%%%%%%%%%%%%%%%%%%%%%%%%%%%%%%%%%%%%%%%%%%%%%%%%%%%%%%%%%%

\subsection{Set-up and notation}\label{notation}

We begin by establishing the notation we will use repeatedly throughout the rest of the paper. Let $K$ be an imaginary quadratic field and let $\mathcal{O}_K$ be the ring of integers of $K$ with discriminant $\Delta_K$. Let $f\geq 1$ be an integer, and let $\mathcal{O}_{K,f}$ be the order of $K$ of conductor $f$, with discriminant $\Delta_Kf^2$. Let $j_{K,f}$ be a CM $j$-invariant associated to $\mathcal{O}_{K,f}$. We define associated constants $\delta$ and $\phi$ as follows:
\begin{itemize}
    \item If $\Delta_Kf^2 \equiv 0 \bmod 4$, or $N$ is odd, let $\delta = \Delta_Kf^2/4$ and $\phi = 0$.
    \item If $\Delta_Kf^2 \equiv 1 \bmod 4$, and $N$ is even, let $\delta = \frac{(\Delta_K - 1)}{4}f^2$, let $\phi =f$.
\end{itemize}
We define matrices in $\GL(2,\Z/N\Z)$ by 
\[
c_\varepsilon = \begin{pmatrix}-\varepsilon&0\\\phi&\varepsilon\end{pmatrix} \quad \text{and} \quad c_{\delta,\phi}(a,b) = \begin{pmatrix}a+b\phi&b\\\delta b&a\end{pmatrix},
\]
where $\varepsilon \in \{\pm 1\}$ and $a,b \in \Z/N\Z$ such that $\det(c_{\delta,\phi}(a,b)) \in (\Z/N\Z)^\times$. We define the Cartan subgroup $\mathcal{C}_{\delta,\phi}(N)$ of $\GL(2,\Z/N\Z)$ by
\[
\mathcal{C}_{\delta,\phi}(N) = \{ c_{\delta,\phi}(a,b) : a,b \in \Z/N\Z, \det(c_{\delta,\phi}(a,b)) \in (\Z/N\Z)^\times \}.
\]
We define the subgroup $\mathcal{N}_{\delta,\phi}(N) = \langle \mathcal{C}_{\delta,\phi}(N), c_1\rangle$ of $\GL(2,\Z/N\Z)$  and sometimes call $\mathcal{N}_{\delta,\phi}(N)$ the ``normalizer" of $\mathcal{C}_{\delta,\phi}(N)$ in $\GL(2,\Z/N\Z)$. In the case that $N$ is prime, $\mathcal{N}_{\delta,\phi}(N)$ is a true normalizer, but it may not be the actual normalizer for composite values of $N$ (see \cite{lozano-galoiscm}, Section 5). Lastly, we write $\mathcal{N}_{\delta,\phi} = \varprojlim \mathcal{N}_{\delta,\phi}(N)$ and regard it as a subgroup of $\GL(2,\widehat{\Z})$.

Let $E/\Q(j_{K,f})$ be an elliptic curve with CM by $\mathcal{O}_{K,f}$, let $N\geq 3$, and let $\rho_{E,N}$ be the Galois representation $\Gal(\overline{\Q(j_{K,f}})/\Q(j_{K,f})) \to \Aut(E[N]) \cong \GL(2,\Z/N\Z)$. Similarly, for a prime $p$, let $\rho_{E,p^\infty}$ be the Galois representation $\Gal(\overline{\Q(j_{K,f}})/\Q(j_{K,f})) \to \Aut(T_p(E)) \cong \GL(2,\Z_p)$. We define the image $G_{E,N} \coloneqq \Gal(\Q(j_{K,f},E[N])/\Q(j_{K,f})) = \operatorname{im} \rho_{E,N}$, and similarly, $G_{E,p^\infty} \coloneqq \operatorname{im} \rho_{E,p^\infty}$.

%%%%%%%%%%%%%%%%%%%%%%%%%%%%%%%%%%%%%%%%%%%%%%%%%%%%%%%%%%%%%%%%

\subsection{Abelian extensions contained in division fields}

Let $F$ be a number field and let $E/F$ be an elliptic curve. We recall some results about abelian extensions contained in $F(E[N])/F$, where $F(E[N])$ denotes the $N$-division field of $E/F$.

\begin{lemma}\label{rootsofunity_in_divfld}
    Let $K$ be an imaginary quadratic field, let $F$ be a number field, let $E/F$ be an elliptic curve with CM by $\mathcal{O}_{K,f}$, and let $N\geq 2$ be an integer. Then, $F(\zeta_N) \subseteq F(E[N])$. 
\end{lemma}

\begin{proof}
    See \cite[III, Corollary 8.1.1]{silverman1}. 
\end{proof}

In particular, if $E$ is defined over $\Q$, then $\Q(\zeta_N) \subseteq \Q(E[N])$ for any $N\geq 2$. In \cite{gonzalez-jimenez-lozano-robledo}, Gonz\'alez-Jim\'enez and Lozano-Robledo prove that $\Q(E[N]) = \Q(\zeta_N)$ only for $N = 2,3,4$, or $5$, and in the CM case, the list shortens to $N= 2$ or $3$. When considering the $2$-division field of an elliptic curve $E/\Q$ with CM, it is possible that there are more roots of unity contained in $\Q(E[2^n])$ for $n\geq 2$. 

\begin{lemma}[\cite{jones2023cm}, Theorem 1.1]\label{moreroots_in_divfld}
    Let $E$ be any elliptic curve over $\Q$ with CM by an order $\mathcal{O}_{K,f}$ in an imaginary quadratic field $K$. Assuming that the discriminant $\Delta_Kf^2$ of $\mathcal{O}_{K,f}$ is even, we have that, for each $n\in \mathbb{N}_{\geq 2}$, that $\Q(\zeta_{2^{n+1}}) \subseteq \Q(E[2^n])$.
\end{lemma}

Depending on whether the discriminant $\Delta_Kf^2$ of $\mathcal{O}_{K,f}$ is even or not, it follows that either $\Q(\zeta_{2^{n+1}})$ or $\Q(\zeta_{2^n})$ is contained in $\Q(E[2^n])$, respectively. For an elliptic curve $E$ with CM by an order in an imaginary quadratic field $K$, we also have that the CM field $K$ is contained in $\Q(E[N])$ for $N \geq 3$.

\begin{lemma}[\cite{bourdon2017torsion}, Lemma 3.15]\label{CMfld_in_divfld}
    Let $K$ be an imaginary quadratic field, let $F$ be a number field, let $E/F$ be an elliptic curve with CM by $\mathcal{O}_{K,f}$, and let $N\geq 3$ be an integer. Then, $K\subseteq F(E[N])$. In particular, if $E$ is defined over $\Q(j_{K,f})$, then $K \subseteq K(j_{K,f}) \subseteq \Q(j_{K,f},E[N])$ for any $N\geq 3$.
\end{lemma}

By Lemmas \ref{rootsofunity_in_divfld} and \ref{CMfld_in_divfld}, we obtain that for an elliptic curve $E/\Q$ with CM, an abelian extension contained in $\Q(E[N])/\Q$ is $K(\zeta_N)$ for $N\geq 3$. In the following section, we outline how to compute another abelian extension contained in $\Q(E[N])/\Q$.

%%%%%%%%%%%%%%%%%%%%%%%%%%%%%%%%%%%%%%%%%%%%%%%%%%%%%%%%%%%%%%%%

\subsection{Preliminary results on computing $\alpha$}\label{computing-alpha}

Up to isomorphism over $\overline{\Q}$, there are thirteen elliptic curves $E/\Q$ with CM. In Table \ref{ell_curves} below (copied from \cite{zywina}), we give an elliptic curve $E_{\Delta_K,f}/\Q$ with each of these thirteen $j$-invariants. 
The curve $E_{\Delta_K,f}$ has complex multiplication by an order $\mathcal{O}_{K,f}$ of $K$ of conductor $f\geq 1$, where the discriminant of $K$ is $\Delta_K$.

{
\renewcommand{\arraystretch}{1.1}

\begin{table}[h]
\begin{center}\begin{tabular}{|c|c|c|l|}\hline  
$j$-invariant & $\Delta_K$ &  $f$ &  Elliptic curve $E_{\Delta_K,f}$ \\ \hline\hline
$0$ &$-3$ & $1$ &  $y^2=x^3+16$ \\  
$2^4 3^3 5^3$ & & $2$ & $y^2=x^3-15x+22$ \\  
 $-2^{15} 3 \cdot 5^3$ & &  $3$ &  $y^2=x^3-480x+4048$  \\ \hline 
$2^6 3^3$ & $-4$ & $1$ &  $y^2=x^3+x$  \\   
 $2^3 3^3 11^3$ &  & $2$ &  $y^2=x^3-11x+14$  \\ \hline 
$-3^3 5^3$ & $-7$ & $1$ &  $y^2=x^3-1715x+33614$  \\  
$3^3 5^3 17^3$ & &  $2$ & $y^2=x^3-29155x+1915998$   \\ \hline 
 $2^6 5^3$ & $-8$ & $1$ & $y^2=x^3-4320x+96768$   \\  \hline 
$-2^{15}$ & $-11$ & $1$ & $y^2=x^3-9504x+365904$  \\ \hline 
$-2^{15} 3^3$ & $-19$ & $1$ & $y^2=x^3-608x+5776$   \\ \hline 
$-2^{18} 3^3 5^3$ & $-43$ & $1$ & $y^2=x^3-13760x+621264$   \\ \hline 
$-2^{15} 3^3 5^3 11^3$ & $-67$ & $1$ & $y^2=x^3-117920x+15585808$  \\ \hline  
$-2^{18} 3^3 5^3 23^3 29^3$ & $-163$ & $1$ & $y^2=x^3-34790720x+78984748304$ \\ \hline
\end{tabular} 
\end{center}
\caption{CM elliptic curves over $\Q$}\label{ell_curves}
\end{table}
}

Recall from Proposition 1.4.(b) of \cite{silverman1} that two elliptic curves are isomorphic over $\overline{\Q}$ if and only if they both have the same $j$-invariant. Let $E/\Q$ be an elliptic curve with CM by $\mathcal{O}_{K,f}$. Then, if $j_{K,f} \neq 0, 1728$, there is a unique, square-free $d\in \Z$ such that $E^d$ is a twist of $E$ by $d$ that gives us the curve $E_{\Delta_K,f}$ in Table \ref{ell_curves} which has the same $j$-invariant as $E$. Note that if $j_{K,f}=0$, then $d\in \Z$ is $6$-th power-free, and if $j_{K,f}=1728$, then $d\in \Z$ is $4$-th power-free. Suppose
\begin{align*}
    E : y^2 &= x^3 + Ax + B, \\
    E_{\Delta_K,f} : y^2 &= x^3 + A'x + B'.
\end{align*}
We will define $\alpha = \alpha(E)$ in terms of the twisting factor $d$, where $\sqrt{\alpha} \in \Q(E[N])$ for some integer $N\geq 1$. Note that in many cases $\alpha = d$, but that is not always the case.

\begin{definition}\label{defn-alpha}
    Given an elliptic curve $E/\Q$ with CM as above, we can define $\alpha = \alpha(E)$ as follows:
\begin{enumerate}
    \item[(1)] If $j(E) \neq 0,1728$, then $E$ is a quadratic twist of $E_{\Delta_K,f}$ by some $d \in \Z$, unique up to squares. In particular, $d^2 = A/A'$ and $d^3 = B/B'$. Let $c$ be the sign of $B/B'$, i.e., $c = \operatorname{sgn}(B/B')$. 
    If $\sqrt{A/A'}$ is an integer, then $d = c\cdot \sqrt{A/A'}$, otherwise $d = c\cdot \sqrt{A/A'}\cdot A' = c\cdot \sqrt{AA'}$. 
    Therefore, $E$ is a quadratic twist of $E_{\Delta_K,f}$ by $d$, and we define $\alpha = d$. 

    \item[(2)] If $j(E) = 0$ ($A=A'=0$), then $E$ is a sextic twist of $E_{-3,1}$ by some $d \in \Z$, unique up to $6$-th powers. In particular, if $B/B' = B/16$ is an integer, then $d = B/16$, otherwise $d = B/16 \cdot 2^6 = 4B$. Therefore, $E$ is a twist of $E_{-3,1}$ by $d$, and we define $\alpha = d$.  

    \item[(3)] If $j(E) = 1728$ ($B=B'=0$), then $E$ is a quartic twist of $E_{-4,1}$ by some $d \in \Z$ unique up to $4$-th powers. In particular, $d = A/A' = A$. Therefore, $E$ is a twist of $E_{-4,1}$ by $d$. If $d = ct^2$ for some square-free integer $t \neq \pm 1, \pm 2$ and $c \in \{\pm 1, \pm 2\}$, then we define $\alpha = t$. Otherwise, we define $\alpha = d$.
\end{enumerate}
\end{definition}

Let $E$ be a quadratic twist of one of the curves in Table \ref{ell_curves}. The following result describes an abelian extension in terms of $\alpha$ that is contained in $\Q(E[N])/\Q$.

\begin{lemma}[\cite{gonzálezjiménez2024modelscmellipticcurves}, Lemma 3.4]\label{lem-twistimage}\label{sqrtd_in_divfld}
    Let $E/F : y^2=x^3+Ax+B$ be an elliptic curve defined over a number field $F$, let $N>2$, and let $G_{E,N}$ be the image of $\rho_{E,N}\colon \Gal(\overline{F}/F)\to \GL(2,\Z/N\Z)$. Let $\alpha \in F$ and let $E^\alpha$ be the quadratic twist of $E$ by $\alpha$, i.e., $E^\alpha : \alpha y^2 = x^3+Ax+B$. Then,
    \begin{enumerate}
        \item $F(E^\alpha[N])\subseteq F(\sqrt{\alpha},E[N])$. In particular, if $\sqrt{\alpha}\in F(E[N])$, then $F(E^\alpha[N])\subseteq F(E[N])$.
        \item If $\sqrt{\alpha}$ does not belong to $F(E[N])$, then $F(E^\alpha[N])=F(\sqrt{\alpha},E[N])$, and $G_{E^\alpha,N}$ is conjugate to $\langle -\operatorname{Id},G_{E,N}\rangle \subseteq \GL(2,\Z/N\Z)$.
    \end{enumerate}
\end{lemma}

Let $\Delta_Kf^2 = 3$ and let $p=3$. The following result describes an abelian extension in terms of $\alpha$ that is contained in $\Q(E[3^n])/\Q$ for an elliptic curve $E$ isomorphic to $E_{-3,1}$.

\begin{lemma}\label{alpha-for-p3-j0}
    Let $E/\Q$ be an elliptic curve with $j(E)=0$ and let $\alpha$ be as in Definition \ref{defn-alpha}.
    Then 
    \begin{itemize}
        \item If $[\mathcal{N}_{-1,1}(3^n) : G_{E^\alpha,3^n}] = 1$ or $3$, then $\Q(\zeta_{3^n},\sqrt{\alpha}) \subseteq \Q(E^\alpha[3^n])$ for $n\geq 1$.
        \item If $[\mathcal{N}_{-1,1}(3^n) : G_{E^\alpha,3^n}] = 2$ or $6$, then $\Q(\zeta_{3^n},\sqrt{\alpha}) = \Q(\zeta_{3^n}) \subseteq \Q(E^\alpha[3^n])$ for $n\geq 1$.
    \end{itemize}
\end{lemma}

\begin{proof}
    Let $E = E_{-3,1}$ from Table \ref{ell_curves}, let $d \in \Q^*$ such that $16d$ is $6$-th-power-free, and let $E^d: y^2 = x^3 + 16d$. By Definition \ref{defn-alpha}, we have that $\alpha = d$. By Proposition 3.12 of \cite{gonzálezjiménez2024modelscmellipticcurves}, it follows that 
    \begin{itemize}
        \item $\Q(E^\alpha[3]) = \Q(\sqrt{\alpha}, \sqrt{-3}, \sqrt[3]{\alpha})$, and
        \item $\Q(E^\alpha[9]) = \Q(\sqrt{-3}, \sqrt[3]{3}, \sqrt{\alpha},  \sqrt[3]{\alpha}, \cos(2\pi/9))$.
    \end{itemize}
    Recall that $\Q(E^\alpha[3]) \subseteq \Q(E^\alpha[9]) \subseteq \Q(E^\alpha[3^n])$ for $n\geq 3$.
    Note that $\sqrt{\alpha} \in \Q(E^\alpha[3])$ and $\Q(E^\alpha[9])$. Therefore, $\Q(\sqrt{\alpha})\subseteq \Q(E^\alpha[3^n])$. 
    By Lemma \ref{rootsofunity_in_divfld}, it follows that $\Q(\zeta_{3^n}) \subseteq \Q(E^\alpha[3^n])$ for all $n\geq 1$. 
    \begin{enumerate}
        \item[(1)] By Proposition 3.12 of \cite{gonzálezjiménez2024modelscmellipticcurves}, we have that
        \begin{itemize}
            \item $[\mathcal{N}_{-1,1}(3) : G_{E^\alpha,3}] = 6$ if and only if $\alpha \in \{1, -27\}$; and
            \item $[\mathcal{N}_{-1,1}(9) : G_{E^\alpha,9}] = 6$ if and only if $\alpha \in \{1, -3, 9, -27, 81, -243\}$.
        \end{itemize}
        Note that in all cases $\sqrt{\alpha} \in \Q(\zeta_3)$, so it follows that $\Q(\zeta_{3}, \sqrt{\alpha}) = \Q(\zeta_{3})$. Therefore, $\Q(\zeta_{3^n},\sqrt{\alpha}) = \Q(\zeta_{3^n}) \subseteq \Q(E^\alpha[3^n])$ for all $n\geq 1$.

        \item[(2)] By Proposition 3.12 of \cite{gonzálezjiménez2024modelscmellipticcurves}, we have that
        \begin{itemize}
            \item $[\mathcal{N}_{-1,1}(3) : G_{E^\alpha,3}] = 3$ if and only if $\alpha = t^3$ for some $t\in \Q^*$, $t\neq 1,-3$; and
            \item $[\mathcal{N}_{-1,1}(9) : G_{E^\alpha,9}] = 3$ if and only if $\alpha = t^3$, or $9t^3$ for some $t\in \Q^*$, $t\neq 1,-3$ and $\alpha = 3t^3$ for some $t \in \Q^*$, $t\neq -1,3$.
        \end{itemize}
        Note that in all cases $\sqrt{\alpha} \not\in \Q(\zeta_3)$. Therefore, $\Q(\zeta_{3^n},\sqrt{\alpha}) \subseteq \Q(E^\alpha[3^n])$ for all $n\geq 1$.

        \item[(3)] By Proposition 3.12 of \cite{gonzálezjiménez2024modelscmellipticcurves}, we have that
        \begin{itemize}
            \item $[\mathcal{N}_{-1,1}(3) : G_{E^\alpha,3}] = 2$ if and only if $\alpha = t^2$, $-3t^2$ for some $t\in \Q^*$, $t\neq \pm 1, \pm 3$, and
            \item $[\mathcal{N}_{-1,1}(9) : G_{E^\alpha,9}] = 2$ if and only if $\alpha = t^2$ or $-3t^2$ for $t \in \Q^*$, $t\neq \pm 1, \pm 3, \pm 9$.
        \end{itemize}
        Note that in all cases $\sqrt{\alpha} \in \Q(\zeta_3)$, so it follows that $\Q(\zeta_3, \sqrt{\alpha}) = \Q(\zeta_3)$. Therefore, $\Q(\zeta_{3^n},\sqrt{\alpha}) = \Q(\zeta_{3^n}) \subseteq \Q(E^\alpha[3^n])$ for all $n\geq 1$.

        \item[(4)] By Proposition 3.12 of \cite{gonzálezjiménez2024modelscmellipticcurves}, we have that
        \begin{itemize}
            \item $[\mathcal{N}_{-1,1}(3) : G_{E^\alpha,3}] = 1$ otherwise; and
            \item $[\mathcal{N}_{-1,1}(9) : G_{E^\alpha,9}] = 1$ otherwise.
        \end{itemize}
        Therefore, $\Q(\zeta_{3^n},\sqrt{\alpha}) \subseteq \Q(E^\alpha[3^n])$ for all $n\geq 1$.
    \end{enumerate}
    Thus, we conclude that for $n\geq 1$, when $[\mathcal{N}_{-1,1}(3^n) : G_{E^\alpha,3^n}] = 1$ or $3$, it follows that $\Q(\zeta_{3^n}, \sqrt{\alpha}) \subseteq \Q(E^\alpha[3^n])$, and when $[\mathcal{N}_{-1,1}(3^n) : G_{E^\alpha,3^n}] = 2$ or $6$, it follows that $\Q(\zeta_{3^n}, \sqrt{\alpha}) = \Q(\zeta_{3^n}) \subseteq ~\Q(E^\alpha[3^n])$. 
\end{proof}

Let $\Delta_Kf^2 = -4$ and let $p=2$. The following result describes an abelian extension in terms of $\alpha$ that is contained in $\Q(E[2^n])/\Q$ for an elliptic curve $E$ isomorphic to $E_{-4,1}$. 
\begin{lemma}\label{alpha-for-p2-j1728}
    Let $E/\Q$ be an elliptic curve with $j(E)=1728$ and let $\alpha$ be as in Definition \ref{defn-alpha}.
    Then
    \begin{itemize}
        \item If $[\mathcal{N}_{-1,0}(2^n) : G_{E^\alpha,2^n}] = 1$ or $2$, then $\Q(\zeta_{2^{n+1}},\sqrt{\alpha}) \subseteq \Q(E^\alpha[2^n])$ for $n\geq 2$.
        \item If $[\mathcal{N}_{-1,0}(2^n) : G_{E^\alpha,2^n}] = 4$, then $\Q(\zeta_{2^{n+1}},\sqrt{\alpha}) = \Q(\zeta_{2^{n+1}}) \subseteq \Q(E^\alpha[2^n])$ for $n\geq 2$.
    \end{itemize}
\end{lemma}

\begin{proof}
    Let $E = E_{-4,1}$ from Table \ref{ell_curves}, let $d$ be a non-zero, $4$-th-power-free, and let $E^d: y^2 = x^3 + d x$. 
    Let $\alpha$ be as in Definition \ref{defn-alpha}. By Corollary 3.1 of \cite{gonzálezjiménez2024modelscmellipticcurves}, it follows that 
    \begin{itemize}
        \item $\Q(E^d[4]) = \Q(\zeta_{8}, \sqrt[4]{d})$, and
        \item $\Q(E^d[8]) = \Q(\zeta_{16}, \sqrt{-1 + \sqrt{2}}, \sqrt[4]{d})$.
    \end{itemize}
    Recall that $\Q(E^d[2]) \subseteq \Q(E^d[4]) \subseteq \Q(E^d[2^n])$ for $n\geq 3$. Observe that $\sqrt[4]{d} \in \Q(E^d[2^n])$ for $n\geq 2$. 
    Since $2$ divides $\Delta_Kf^2 = -4$, by Lemma \ref{moreroots_in_divfld}, it follows that $\Q(\zeta_{2^{n+1}}) \subseteq \Q(E^d[2^n])$ for all $n\geq 2$. 

    \begin{enumerate}
        \item[(1)] By Lemma 3.3 of \cite{gonzálezjiménez2024modelscmellipticcurves}, we have that $[\mathcal{N}_{-1,0}(8) : G_{E^\alpha,8}] = 4$ if and only if $d \in \{\pm 1, \pm 2, \pm 4, \pm 8\}$.
        By Definition \ref{defn-alpha}, if $d \in \{\pm1, \pm 2\}$, then $\alpha = d$, and if $d \in \{\pm 4, \pm 8\}$, then $\alpha = \sqrt{2}$. 
        Note that in all cases $\sqrt{\alpha} \in \Q(\zeta_8)$, so it follows that $\Q(\zeta_{8}, \sqrt{\alpha}) = \Q(\zeta_{8})$. Therefore, $\Q(\zeta_{2^{n+1}},\sqrt{\alpha}) = \Q(\zeta_{2^{n+1}}) \subseteq \Q(E^\alpha[2^n])$ for all $n\geq 1$.

        \item[(2)] By Lemma 3.3 of \cite{gonzálezjiménez2024modelscmellipticcurves}, we have that $[\mathcal{N}_{-1,0}(8) : G_{E^\alpha,8}] = 2$ if and only if $d = \pm t^2$ or $\pm 2 t^2$ for some other square-free integer $t\neq \pm 1,\pm 2$. 
        By Definition \ref{defn-alpha}, we have that $\alpha = t$. 
        Recall that from Corollary 3.1 of \cite{gonzálezjiménez2024modelscmellipticcurves} we have $\sqrt[4]{d} \in \Q(E^d[2^n])$ for $n\geq 2$. We will show that $\sqrt{\alpha} \in \Q(E^d[2^n])$, but $\sqrt{\alpha}\not\in \Q(\zeta_{2^{n+1}})$ for $n \geq 2$.
        
        \begin{itemize}
            \item First we will handle the case where $n=2$.
            
            If $d = \pm t^2$, then $\sqrt[4]{d} = \sqrt[4]{\pm t^2} = \sqrt[4]{\pm 1}\sqrt{t}$, where $\sqrt[4]{\pm 1} \in \Q(\zeta_8)$ and $\sqrt{t} \not\in \Q(\zeta_8)$. Therefore, $\Q(E^d[4]) = \Q(\zeta_8, \sqrt{t})$ and is abelian over $\Q$. 
            
            If $d = \pm 2t^2$, then $\sqrt[4]{d} = \sqrt[4]{\pm 2 t^2} = \sqrt[4]{\pm 2}\sqrt{t}$, where $\sqrt[4]{\pm 2}, \sqrt{t}\not\in \Q(\zeta_{8})$. Note $\Q(E^d[4])/\Q$ is not abelian, but $\sqrt{t} \in \Q(E^d[4])$ and $\Q(\zeta_8, \sqrt{t}) \subseteq \Q(E^d[4])$ is abelian over $\Q$. 
            
            Note that $\alpha = t$ and in all cases $\sqrt{\alpha} \not\in \Q(\zeta_8)$; however, $\sqrt{\alpha} \in \Q(E^\alpha[4])$. 

            \item Let $n\geq 3$ and let $d = \pm 2^j t^2$, where $j\in \{0,1\}$. 

            Observe that $\sqrt[4]{d} = \sqrt[4]{\pm 2^j t^2} = \sqrt[4]{\pm 2^j} \sqrt{t}$, where $\sqrt{t} \not\in \Q(\zeta_{16})$. From the proof of Lemma 3.3 in \cite{gonzálezjiménez2024modelscmellipticcurves}, it follows that $\sqrt[4]{\pm 2^j} \in \Q(\zeta_{16}, \sqrt{-1 + \sqrt{2}}) \subseteq \Q(E^d[8])$. Therefore, $\Q(E^d[8]) = \Q(\zeta_{16},\sqrt{-1 + \sqrt{2}}, \sqrt[4]{\pm 2^j}\sqrt{t}) = \Q(\zeta_{16},\sqrt{-1 + \sqrt{2}}, \sqrt{t})$, and hence, $\sqrt{t} \in \Q(E^d[8])$.

            Note that $\alpha = t$ and in all cases $\sqrt{\alpha} \not\in \Q(\zeta_{16})$; however, $\sqrt{\alpha} \in \Q(E^\alpha[8])$. 
        \end{itemize}
        
        Therefore, we conclude that $\Q(\zeta_{2^{n+1}},\sqrt{\alpha}) \subseteq \Q(E^\alpha[2^n])$ for all $n\geq 2$.

        \item[(3)] By Lemma 3.3 of \cite{gonzálezjiménez2024modelscmellipticcurves}, we have that $[\mathcal{N}_{-1,0}(8) : G_{E^\alpha,8}] = 1$ otherwise. By Definition \ref{defn-alpha}, we have that $\alpha = d$, and since $\sqrt[4]{\alpha} \in \Q(E^\alpha[2^n])$ for $n\geq 2$, we have that $\sqrt{\alpha} \in \Q(E^\alpha[2^n])$. Therefore, $\Q(\zeta_{2^{n+1}}, \sqrt{\alpha}) \subseteq \Q(E^\alpha[2^n])$ for all $n\geq 1$.
    \end{enumerate}
    Thus, we conclude that for $n\geq 2$, when $[\mathcal{N}_{-1,0}(2^n) : G_{E^\alpha,2^n}] = 1$ or $2$, it follows that $\Q(\zeta_{2^{n+1}}, \sqrt{\alpha}) \subseteq \Q(E^\alpha[2^n])$, and when $[\mathcal{N}_{-1,0}(2^n) : G_{E^\alpha,2^n}] = 4$, it follows that $\Q(\zeta_{2^{n+1}}, \sqrt{\alpha}) = \Q(\zeta_{2^{n+1}}) \subseteq \Q(E^\alpha[2^n])$. 
\end{proof}

Let $\Delta_Kf^2 = -8$ or $-16$ and let $p=2$. The following result describes an abelian extension in terms of $\alpha$ that is contained in $\Q(E[2^n])/\Q$ for an elliptic curve $E$ isomorphic to $E_{-8,1}$ or $E_{-4,2}$. 
\begin{lemma}\label{alpha-for-p2-8-16}
    Let $E = E_{-8,1}$ or $E_{-4,2}$ from Table \ref{ell_curves}, and let $\alpha$ be as in Definition \ref{defn-alpha}. Then
    \begin{itemize}
        \item If $[\mathcal{N}_{\delta,0}(2^n) : G_{E^\alpha,2^n}] = 1$, then $\Q(\zeta_{2^{n+1}},\sqrt{\alpha}) \subseteq \Q(E^\alpha[2^n])$ for $n\geq 2$.
        \item If $[\mathcal{N}_{\delta,0}(2^n) : G_{E^\alpha,2^n}] = 2$, then $\Q(\zeta_{2^{n+1}},\sqrt{\alpha}) = \Q(\zeta_{2^{n+1}}) \subseteq \Q(E^\alpha[2^n])$ for $n\geq 2$.
    \end{itemize}
\end{lemma}

\begin{proof}
    Let $E = E_{-8,1}$ or $E_{-4,2}$ from Table \ref{ell_curves}, and let $\alpha$ be as in Definition \ref{defn-alpha}. If $E = E_{-8,1}$, then by Proposition 3.6 of \cite{gonzálezjiménez2024modelscmellipticcurves} it follows that 
    \begin{itemize}
        \item $\Q(E^\alpha[2]) = \Q(\sqrt{2})$, 
        \item $\Q(E^\alpha[4]) = \Q(\beta, \sqrt{\alpha}\sqrt[4]{2})$ such that $f_4(\beta)=0$, where $f_4(x) = x^8 + 6x^4 + 1$, and 
        \item $\Q(E^\alpha[8]) = \Q(\gamma,\sqrt{\alpha})$ such that $f_8(\gamma)=0$, where $f_8(x)=x^{32} + 16x^{31} + 128x^{30} + 672x^{29} +  2544x^{28} + 7200x^{27} + 15352x^{26} + 24272x^{25} + 26904x^{24} + 17312x^{23} - 304x^{22} - 11984x^{21} - 9672x^{20} - 2720x^{19} - 3592x^{18} - 7552x^{17} -     2798x^{16} + 6224x^{15} + 6368x^{14} - 672x^{13} - 2224x^{12} + 3360x^{11} +     4952x^{10} - 1072x^9 - 4600x^8 - 1120x^7 + 1776x^6 + 752x^5 - 264x^4 -  96x^3 + 24x^2 + 1$.
    \end{itemize}
    If $E = E_{-4,2}$, then by Proposition 3.9 of \cite{gonzálezjiménez2024modelscmellipticcurves} it follows that 
    \begin{itemize}
        \item $\Q(E^\alpha[2]) = \Q(\sqrt{2})$, 
        \item $\Q(E^\alpha[4]) = \Q(\beta, \sqrt{\alpha})$ such that $f_4(\beta)=0$, where $f_4(x) = x^8 - 4x^6 + 8x^4 - 4x^2 + 1$, and 
        \item $\Q(E^\alpha[8]) = \Q(\gamma,\sqrt{\alpha})$ such that $f_8(\gamma)=0$, where $f_8(x)=x^{32} + 16x^{31} + 120x^{30} + 560x^{29} + 1848x^{28} + 4784x^{27} + 11000x^{26} + 25344x^{25} + 59844x^{24} + 133856x^{23} +  260768x^{22} + 419392x^{21} + 534920x^{20} + 513536x^{19} + 332032x^{18} + 93856x^{17} - 43548x^{16} - 22112x^{15} + 61056x^{14} + 77728x^{13} + 20768x^{12} - 18304x^{11} + 320x^{10} + 21440x^{9} + 8240x^{8} - 8256x^{7} - 1888x^{6} + 3584x^{5} + 800x^{4} - 1216x^{3} + 320x^{2} + 8$.
    \end{itemize}
    Recall that $\Q(E^\alpha[2]) \subseteq \Q(E^\alpha[4]) \subseteq \cdots \subseteq \Q(E^\alpha[2^n])$ for $n\geq 3$.
    Note that in both cases, $\Q(E^\alpha[2]) = \Q(\sqrt{2})$ and that $\sqrt{\alpha} \in \Q(E^\alpha[4])$ and $\Q(E^\alpha[8])$. Therefore, $\Q(\sqrt{\alpha}) \subseteq \Q(E^\alpha[2^n])$ for all $n\geq 2$.
    Since $2$ divides $\Delta_Kf^2$, by Lemma \ref{moreroots_in_divfld}, it follows that $\Q(\zeta_{2^{n+1}}) \subseteq \Q(E^\alpha[2^n])$ for all $n\geq 1$. 
    
    \begin{enumerate}
        \item[(1)] By Propositions 3.6 and 3.9 of \cite{gonzálezjiménez2024modelscmellipticcurves}, we have that $[\mathcal{N}_{\delta,0}(2^n) : G_{E^\alpha,2^n}] = 2$ if and only if $\alpha \in \{\pm 1, \pm 2\}$. Note that in all cases $\sqrt{\alpha} \in \Q(\zeta_8)$, so it follows that $\Q(\zeta_8, \sqrt{\alpha}) = \Q(\zeta_8)$. Therefore, $\Q(\zeta_{2^{n+1}},\sqrt{\alpha}) = \Q(\zeta_{2^{n+1}}) \subseteq \Q(E^\alpha[2^n])$ for all $n\geq 1$.

        \item[(2)] By Propositions 3.6 and 3.9 of \cite{gonzálezjiménez2024modelscmellipticcurves}, we have that $[\mathcal{N}_{\delta,0}(2^n) : G_{E^\alpha,2^n}] = 1$ otherwise. Therefore, $\Q(\zeta_{2^{n+1}},\sqrt{\alpha}) \subseteq \Q(E^\alpha[2^n])$ for all $n\geq 1$.
    \end{enumerate}
    Thus, we conclude that for $n\geq 2$, when $[\mathcal{N}_{\delta,0}(2^n) : G_{E^\alpha,2^n}] = 1$, it follows that $\Q(\zeta_{2^{n+1}}, \sqrt{\alpha}) \subseteq \Q(E^\alpha[2^n])$, and when $[\mathcal{N}_{-1,0}(2^n) : G_{E^\alpha,2^n}] = 2$, it follows that $\Q(\zeta_{2^{n+1}}, \sqrt{\alpha}) = \Q(\zeta_{2^{n+1}}) \subseteq \Q(E^\alpha[2^n])$. 
\end{proof}

%%%%%%%%%%%%%%%%%%%%%%%%%%%%%%%%%%%%%%%%%%%%%%%%%%%%%%%%%%%%%%%%

\subsection{Classification of $\ell$-adic images of Galois representations}\label{classification-of-images}

First, we recall results from \cite{lozano-galoiscm} that provide details about the image of Galois representations attached to elliptic curves with CM. 

\begin{theorem}[\cite{lozano-galoiscm}, Theorems 1.1 and 1.2]\label{gal-repns-att-to-ec-w-cm}
    Let $E/\Q(j_{K,f})$ be an elliptic curve with CM by $\mathcal{O}_{K,f}$, let $N\geq 3$, and let $\rho_{E,N}$ be the Galois representation $\Gal(\overline{\Q(j_{K,f}})/\Q(j_{K,f})) \to \Aut(E[N]) \cong \GL(2,\Z/N\Z)$, and let $\rho_{E}$ be the Galois representation $\Gal(\overline{\Q(j_{K,f}})/\Q(j_{K,f})) \to \varprojlim \Aut(E[N]) \cong \GL(2,\widehat{\Z})$. Then, 
    \begin{itemize}
        \item[(1)] There is a $\Z/N\Z$-basis of $E[N]$ such that the image of $\rho_{E,N}$ is contained in $\mathcal{N}_{\delta,\phi}(N)$, and the index of the image of $\rho_{E,N}$ in $\mathcal{N}_{\delta,\phi}(N)$ is a divisor of the order of $\mathcal{O}_{K,f}^\times/\mathcal{O}_{K,f,N}^\times$, where $\mathcal{O}_{K,f,N}^\times = \{u \in \mathcal{O}_{K,f}^\times : u \equiv 1 \bmod N\mathcal{O}_{K,f}\}$.  
        \item[(2)] There is a compatible system of bases of $E[N]$ such that the image of $\rho_E$ is contained in $\mathcal{N}_{\delta,\phi}$, and the index of the image of $\rho_E$ in $\mathcal{N}_{\delta,\phi}$ is a divisor of the order $\mathcal{O}_{K,f}^\times$. In particular, the index is a divisor of $4$ or $6$. 
    \end{itemize}
\end{theorem}

One of our goals is to understand what the commutator subgroups of $\mathcal{N}_{\delta,\phi}(N)$ or subgroups of $\mathcal{N}_{\delta,\phi}(N)$ are for images of Galois representations. We can start by recalling information about a subgroup of $\mathcal{N}_{\delta,\phi}(N)$ that we know is abelian, the Cartan subgroup. 

\begin{lemma}
    Let $N\geq 2$. The group $\mathcal{C}_{\delta,\phi}(N)$ is isomorphic to $(\mathcal{O}_{K,f}/N\mathcal{O}_{K,f})^\times$ so, in particular, it is an abelian subgroup of $\GL(2,\Z/N\Z)$.
\end{lemma}

\begin{proof}
    See \cite{lozano-galoiscm}, Lemma 2.5 and Remark 2.6. 
\end{proof}

Next, we recall a result about the size of Cartan sugbgroups of $\GL(2,\Z/N\Z)$, which we will use repeatedly throughout the paper to compute the size of $\mathcal{N}_{\delta,\phi}(N)$.

\begin{lemma}[\cite{lozano-galoiscm}, Lemma 2.5]\label{size_of_cartan}
    Let $f\geq 1$, let $K$ be an imaginary quadratic field of discriminant $\Delta_K$, and let $\mathcal{O}_{K,f}$ be the order of $K$ of conductor $f$ with discriminant $\Delta_Kf^2$. Let $N=p^n$ for some prime $p$ and $n\geq 1$, and let $\mathcal{C}_{\delta,\phi}(N) \cong (\mathcal{O}_{K,f}/N\mathcal{O}_{K,f})^\times$, where $\delta$ and $\phi$ are as defined earlier. Then, 
    \begin{itemize}
        \item[(1)] If $p\mid \Delta_Kf$, then $|\mathcal{C}_{\delta,\phi}(p^n)| = p^{2n-1}(p-1)$.
        \item[(2)] If $\gcd(p,f)=1$ and $p$ is split in $K$, then $|\mathcal{C}_{\delta,\phi}(p^n)| = p^{2(n-1)}(p-1)^2$.
        \item[(3)] If $\gcd(p,f)=1$ and $p$ is inert in $K$, then $|\mathcal{C}_{\delta,\phi}(p^n)| = p^{2(n-1)}(p^2-1)$.
    \end{itemize}
    Finally, in all cases $|\mathcal{C}_{\delta,\phi}(p^{n+1})|/|\mathcal{C}_{\delta,\phi}(p^n)| = p^2$.
\end{lemma}

Note that $\mathcal{C}_{\delta,\phi}(N)$ is an index $2$ subgroup of $\mathcal{N}_{\delta,\phi}(N)$. Therefore, $|\mathcal{N}_{\delta,\phi}(N)| = 2\cdot |\mathcal{C}_{\delta,\phi}(N)|$.

%%%%%%%%%%%%%%%%%%%%%%%%%%%%%%%%%%%%%%%%%%%%%%%%%%%%%%%%%%%%%%%%

\subsection{Preliminary results about commutator subgroups of $\ell$-adic images}

In \cite{lozano-galoiscm}, Lozano-Robledo explicitly describes the groups of $\GL(2,\Z_p)$, up to conjugation, that can occur as images $G_{E,p^\infty}$ for an elliptic curve $E/\Q(j_{K,f})$ with CM by $\mathcal{O}_{K,f}$. The following preliminary results will be used later to bound the sizes of certain subgroups of $\GL(2,\Z/p^n\Z)$.

\begin{lemma}\label{ker_of_N}
    Let $p$ be a prime, let $n\geq 1$ be an integer, and let $\pi : \mathcal{N}_{\delta,\phi}(p^{n+1}) \to \mathcal{N}_{\delta,\phi}(p^n)$ be the natural reduction map on the normalizer of Cartan. Then $|\ker(\pi)| = p^2$.
\end{lemma}

\begin{proof}
    Since $\pi: \mathcal{N}_{\delta,\phi}(p^{n+1}) \to \mathcal{N}_{\delta,\phi}(p^n)$ is surjective, it follows that $|\ker(\pi)| \ = |\mathcal{N}_{\delta,\phi}(p^{n+1})| / |\mathcal{N}_{\delta,\phi}(p^n)|$. By Lemma 2.5 of \cite{lozano-galoiscm}, we know that the size of the normalizer grows by a factor of $p^2$ at every level, i.e., $|\ker(\pi)| = p^2$. 
    Note that $\mathcal{N}_{\delta,\phi}(p^{n+1}) = \left\{ \mathcal{C}_{\delta,\phi}(p^{n+1}), \ c_\varepsilon\cdot \mathcal{C}_{\delta,\phi}(p^{n+1})\right\}$, so it suffices to check if the following matrices are in $\ker(\pi)$:
    \begin{itemize}
        \item Suppose $c_{\delta,\phi}(a,b)\cdot c_\varepsilon  \bmod p^{n+1} \ \equiv \ I \bmod p^n$. Then $-a \bmod p^{n+1} \equiv 1 \bmod p^n$, which contradicts that $a \bmod p^{n+1} \equiv 1 \bmod p^n$. Therefore, $c_{\delta,\phi}(a,b)\cdot c_\varepsilon \not\in \ker(\pi)$.
        \item Suppose $c_{\delta,\phi}(a,b) \bmod p^{n+1} \ \equiv \ I \bmod p^n$. Since $a \bmod p^{n+1} \equiv 1 \bmod p^n$, there are $p^{n+1}/p^n=p$ choices for $a$, namely $1 + k_1\cdot p^n$, where $0 \leq k_1 \leq p-1$. Similarly, since $b \bmod p^{n+1} \equiv 0 \bmod p^n$, there are $p^{n+1}/p^n=p$ choices for $b$, namely $k_2 \cdot p^n$, where $0\leq k_2 \leq p-1$.  It follows that $a+b\phi \bmod p^{n+1} \equiv 1 \bmod p^n$ and that $\delta b \bmod p^{n+1} \equiv 0 \bmod p^n$. Therefore, $c_{\delta,\phi}(a,b) \in \ker(\pi)$, and there are $p\cdot p = p^2$ such matrices.
    \end{itemize}
    Thus, we have shown that $\ker(\pi)$ has order $p^2$, and we have found $p^2$ matrices that are contained in $\ker(\pi)$. We conclude that the kernel of the natural reduction of the normalizer of Cartan is
    \begin{align*}
        \ker(\pi) &= \left\{\begin{pmatrix}1 + k_1\cdot p^n + (k_2\cdot p^n)\phi & k_2\cdot p^n\\ \delta(k_2\cdot p^n) & 1 + k_1\cdot p^n\end{pmatrix}: 0 \leq k_1, k_2 \leq p-1 \right\}, \\
        &= \left\{ \begin{pmatrix}1&0\\0&1\end{pmatrix} + p^n \begin{pmatrix}k_1 + \phi k_2& k_2\\ \delta k_2 & k_1\end{pmatrix}: 0 \leq k_1, k_2 \leq p-1\right\}.
    \end{align*}
\end{proof}

\begin{remark}
    The level of definition of an image is the smallest level at which every level above it is the full inverse image into the normalizer of Cartan. In particular, the full inverse image of the identity matrix is in the pull back, so if the level of definition of an image is $p^k$ for $k\geq 1$, then the full kernel of reduction mod $p^k$ (and above) is contained in the image mod $p^k$ (and above). See Proposition 12.1.4 of \cite{elladic} for the level of definition of an image attached to an elliptic curve $E/\Q(j(E))$ with CM. 
\end{remark}

The following result tells us that the kernel of the natural reduction map on $\mathcal{N}_{\delta,\phi}(p^{n+1})$ is the same as the kernel of the natural reduction map on $G_{E,p^{n+1}}$, after the level of definition of $G_{E,p^n}$.

\begin{lemma}\label{ker_of_G_equals_N}
    Let $p$ be a prime and let $G_{E,p^\infty}$ be any of the possible $p$-adic images. For $n\geq 1$, let $\pi|_{G}: G_{E,p^{n+1}} \to G_{E,p^n}$ be the restriction of $\pi$ to the image $G_{E,p^{n+1}}$. For $k \geq 1$, let $p^k$ be the level of definition of the image $G_{E,p^n}$. Then for $n\geq k$, we have that $\ker(\pi|_{G}) = \ker(\pi)$.
\end{lemma}

\begin{proof}
    Let $\pi|_{G} : G_{E,p^{n+1}} \to G_{E,p^n}$ be the restriction of $\pi$ to the image $G_{E,p^{n+1}}$ for all $n\geq 1$, and for $k\geq 1$ let $p^k$ be the level of definition of $G_{E,p^n}$. The kernel of $\pi|_{G}$ contains all of the matrices in $G_{E,p^{n+1}}$ that map to the identity mod $p^n$. This is a subgroup of the matrices in $\mathcal{N}_{\delta,\phi}(p^{n+1})$ that map to the identity mod $p^n$, so $\ker(\pi|_{G}) \subseteq \ker(\pi)$. By the level of definition, the full inverse image of the identity matrix mod $p^k$ is the kernel of reduction mod $p^k$, so it follows that $\ker(\pi) \subseteq \ker(\pi|_{G})$ for $n\geq k$. Therefore, we conclude that, after the level of definition, $\ker(\pi) = \ker(\pi|_{G})$.
\end{proof}

Next, we prove that the natural reduction map on the commutator subgroup of an image of a Galois representation is surjective.

\begin{lemma}\label{pi_comm_surj}
    Let $p$ be a prime, let $G_{E,p^\infty}$ be any $p$-adic image, and let $G'_{E,p^\infty}$ be the commutator subgroup of $G_{E,p^\infty}$. Let $\pi|_{G'} \colon G'_{E,p^{n+1}} \to G'_{E,p^n}$ be the restriction of $\pi$ to the commutator subgroup. Then $\pi|_{G'}$ is surjective for all $n\geq 1$. 
\end{lemma}

\begin{proof}
    Recall that $\pi|_G \colon G_{E,p^{n+1}} \to G_{E,p^n}$ the restriction of the natural reduction map $\pi$ to the image $G_{E,p^{n+1}}$ for all $n\geq 1$ is surjective. Let $\pi|_{G'} \colon G'_{E,p^{n+1}} \to G'_{E,p^n}$ be the restriction of $\pi$ to the commutator subgroup $G'_{E,p^{n+1}}$. Recall that the definition of the commutator subgroup is 
    \begin{align*}
        G'_{E,p^n} = \{ABA^{-1}B^{-1} : A,B \in G_{E,p^n}\}.
    \end{align*}
    Let $\gamma \in G'_{E,p^n}$ such that $\gamma = ABA^{-1}B^{-1}$ for some $A,B \in G_{E,p^n}$. Since $\pi|_{G}$ is surjective, we know there exist $A',B' \in G_{E,p^{n+1}}$ such that $\pi|_{G}(A') = A$ and $\pi|_{G}(B') = B$. By definition of a commutator subgroup, let $\gamma' = A'B'(A')^{-1}(B')^{-1} \in G'_{E,p^{n+1}}$. Since $\pi$ is a homomorphism and $\pi|_{G}$ and $\pi|_{G'}$ act the same on elements of $G'_{E,p^{n+1}}$, we have the following:
    \begin{align*}
        \pi|_{G'}(\gamma') &= \pi|_{G'}(A'B'(A')^{-1}(B')^{-1}), \\
        &= \pi|_{G}(A'B'(A')^{-1}(B')^{-1}), \\
        &= \pi|_{G}(A')\cdot \pi|_{G}(B')\cdot \pi|_{G}((A')^{-1})\cdot \pi|_{G}((B')^{-1}), \\
        &= ABA^{-1}B^{-1}, \\
        &= \gamma \in G'_{E,p^n}.
    \end{align*}
    Therefore, $\pi|_{G'}$ is surjective. 
\end{proof}

The following lemma tells us about the size of the kernel of the natural reduction map on the commutator subgroup of an image of a Galois representation. 

\begin{lemma}\label{comm_ker_size}
    Let $p$ be a prime, let $G_{E,p^\infty}$ be any of the possible $p$-adic images, and let $G'_{E,p^\infty}$ be the commutator subgroup of $G_{E,p^\infty}$. Let $\pi|_{G'}\colon G'_{E,p^{n+1}} \to G'_{E,p^n}$ be the restriction of $\pi$ to the commutator subgroup. Then  $|\ker(\pi|_{G'})| = 1,p,$ or $p^2$.
\end{lemma}

\begin{proof}
    Let $\pi|_{G'}\colon G'_{E,p^{n+1}} \to G'_{E,p^n}$ be the restriction of $\pi$ to the commutator subgroup $G'_{E,p^{n+1}}$. Since $\ker(\pi|_{G'}) \subseteq \ker(\pi)$ and $|\ker(\pi)|=p^2$ by Lemma \ref{ker_of_N}, it follows that $|\ker(\pi|_{G'})|$ divides $p^2$. 
\end{proof}

%%%%%%%%%%%%%%%%%%%%%%%%%%%%%%%%%%%%%%%%%%%%%%%%%%%%%%%%%%%%%%%%%%%%%%%%%%%%%%%%

\section{Computing lower bounds for commutator subgroups}\label{main_method}

In this section, we will outline how to compute a lower bound for the commutator subgroup of $G_{E,p^n}$. Let $\pi|_{G'} \colon G'_{E,p^{n+1}} \to G'_{E,p^n}$ be the restriction of $\pi$ to the commutator subgroup $G'_{E,p^{n+1}}$. By Lemma \ref{pi_comm_surj}, the map $\pi|_{G'}$ is surjective, so it follows that for $n\geq 1$,
\[
    |G'_{E,p^{n+1}}| \ \geq \ |G'_{E,p^{n}}|\cdot |\ker(\pi|_{G'})|,
\]
where this is an inequality since we will only provide a lower bound for $|\ker(\pi|_{G'})|$. By Lemma \ref{comm_ker_size}, we know that $|\ker(\pi|_{G'})|$ divides $p^2$. If we can construct an element of $G'_{E,p^{n+1}}$ that is a non-trivial element of $\ker(\pi|_{G'})$, then we can show that $|\ker(\pi|_{G'})|$ is bounded below by $p$. 

One can compute elements of $\ker(\pi|_{G'})$ as commutators as follows. Recall that $G_{E,p^\infty}$ is generated by Cartan matrices, of the form $c_{\delta,\phi}(a,b)$, and complex conjugation, of the form $c_\varepsilon$. Let
\[
A_{p^{n+1}} = c_\varepsilon \bmod p^{n+1} \quad \text{and} \quad B_{p^{n+1}} = c_{\delta,\phi}(a,b) \bmod p^{n+1},
\]
where $c_{\delta,\phi}(a,b)$ is a generator of $G_{E,p^\infty}$, and $A_{p^{n+1}} \equiv A_{p^n} \bmod p^n$ and $B_{p^{n+1}} \equiv B_{p^n} \bmod p^n$. Using $A_{p^{n+1}}$ and $B_{p^{n+1}}$, we can compute the following commutator element 
\[
Y_{p^{n+1}} \ = \ A_{p^{n+1}}B_{p^{n+1}}(A_{p^{n+1}})^{-1}(B_{p^{n+1}})^{-1} \ \in \ G'_{E,p^{n+1}}.
\]
Suppose that the level of definition of $G_{E,p^n}$ is $p^k$ for some $k\geq 1$. Then for $n\geq k$, the commutator subgroup $G'_{E,p^{n+1}}$ is the full inverse image of $G'_{E,p^n}$ since $\pi|_{G'}$ is surjective.
Therefore, we can multiply $A_{p^{n+1}}$ and $B_{p^{n+1}}$ by an element in $\ker(\pi|_{G})$ since these products will still reduce to $A_{p^n}$ and $B_{p^n}$ in $G_{E,p^n}$, respectively. From Lemma \ref{ker_of_N} and Lemma \ref{ker_of_G_equals_N}, we know precisely which matrices are in $\ker(\pi|_G)$. Take $\kappa \in \ker(\pi|_{G})$ and let 
\[
A_{p^{n+1}}' = A_{p^{n+1}} \cdot \kappa \quad \text{and} \quad B_{p^{n+1}}' = B_{p^{n+1}}\cdot \kappa.
\]
Note that $A_{p^{n+1}}'$ reduces to $A_{p^n} = c_\varepsilon \bmod p^n$, but is not $A_{p^{n+1}} = c_\varepsilon \bmod p^{n+1}$. Using $A_{p^{n+1}}'$ and $B_{p^{n+1}}'$, we can compute the following commutator element
\[
Y_{p^{n+1}}' \ = \ A_{p^{n+1}}'B_{p^{n+1}}'(A_{p^{n+1}}')^{-1}(B_{p^{n+1}}')^{-1} \ \in \ G'_{E,p^{n+1}}.
\]
Taking the quotient of commutators, it follows that $Y_{p^{n+1}}/Y_{p^{n+1}}'$ is a commutator in $G'_{E,p^{n+1}}$. By construction, $Y_{p^{n+1}}/Y_{p^{n+1}}'$ lies in the intersection $G'_{E,p^{n+1}} \cap \ker(\pi|_{G})$. Thus, $Y_{p^{n+1}}/Y_{p^{n+1}}' \in \ker(\pi|_{G'})$. If $Y_{p^{n+1}}/Y_{p^{n+1}}'$ is a non-trivial element in $\ker(\pi|_{G'})$, then $|\ker(\pi|_{G'})| \ \geq p$. 

\begin{remark}
    This method relies on the choice of $\kappa$, which there are only finitely many of. If $Y_{p^{n+1}}/Y_{p^{n+1}}'$ is trivial, then one can try another choice of $\kappa \in \ker(\pi|_{G})$. 
\end{remark}

%%%%%%%%%%%%%%%%%%%%%%%%%%%%%%%%%%%%%%%%%%%%%%%%%%%%%%%%%%%%%%%%%%%%%%%%%%%%%%%%

\section{Proofs of the Main Results}\label{proofs_main_results}

In this section, we prove Theorem \ref{pndivDeltaKf} in Section \ref{sect-pndivdisc}, Theorems \ref{pdivDeltaKf} and \ref{3divDeltaKf} in Section \ref{sect-pdivdisc}, and Theorem \ref{casep2} in Section \ref{sect-casep2}. 

\begin{remark}\label{Magma_verified}
    All computations done in this section have been verified using \verb|Magma| \cite{Magma}. The \verb|Magma| code used can be found on the author's GitHub repository \cite{GitHubPaperCode}. 
\end{remark}

%%%%%%%%%%%%%%%%%%%%%%%%%%%%%%%%%%%%%%%%%%%%%%%%%%%%%%%%%%%%%%%%%%%%%%%%%%%%%%%%

\subsection{Odd primes not dividing the discriminant}\label{sect-pndivdisc}

Let $p$ be an odd prime not dividing $2\Delta_Kf$. In the cases where $p$ is split or inert, we can define the Cartan subgroup as follows. If $p$ is split in $K$, then we can define the \textit{split Cartan subgroup} $\mathcal{C}_{\mathrm{sp}}(p^n)$ as
\[
    \mathcal{C}_{\mathrm{sp}}(p^n) \ = \ \left\{ c_{\mathrm{sp}}(a,b) = \begin{pmatrix}a&0\\0&b\end{pmatrix} : a,b \in \Z/p^n\Z, ab \neq 0 \right\} \subseteq \GL(2,\Z/p^n\Z),
\]
where the normalizer of the split Cartan subgroup $\mathcal{N}_{\mathrm{sp}}(p^n)$ is 
\[
    \mathcal{N}_{\mathrm{sp}}(p^n) \ = \ \left\langle \gamma_1 = \begin{pmatrix}0&1\\1&0\end{pmatrix}, \mathcal{C}_{\mathrm{sp}}(p^n)\right\rangle \subseteq \GL(2,\Z/p^n\Z).
\]
Note that the size of the split Cartan subgroup is given in Lemma \ref{size_of_cartan}.(2). If $p$ is inert in $K$, then we can define the \textit{non-split Cartan subgroup} $\mathcal{C}_{\mathrm{ns}}(p^n)$ as
\[
    \mathcal{C}_{\mathrm{ns}}(p^n) \ = \ \left\{ c_{\mathrm{ns}}(a,b) =\begin{pmatrix}a&\delta b\\b&a\end{pmatrix} : a,b \in \Z/p^n\Z, (a,b) \neq (0,0) \right\} \subseteq \GL(2,\Z/p^n\Z),
\]
where $\delta$ is defined as before and the normalizer of the non-split Cartan subgroup $\mathcal{N}_{\mathrm{ns}}(p^n)$ is 
\[
    \mathcal{N}_{\mathrm{ns}}(p^n) \ = \ \left\langle c_{1} = \begin{pmatrix}1&0\\0&-1\end{pmatrix}, \mathcal{C}_{\mathrm{sp}}(p^n)\right\rangle \subseteq \GL(2,\Z/p^n\Z).
\]
Note that the size of the non-split Cartan subgroup is given in Lemma \ref{size_of_cartan}.(3). 

We proceed with the proof of Theorem \ref{pndivDeltaKf}.

\begin{proof}[Proof of Theorem \ref{pndivDeltaKf}]
    Let $E/\Q$ be an elliptic curve with CM by an order $\mathcal{O}_{K,f}$, $f\geq 1$. Let $p$ be an odd prime not dividing $\Delta_Kf$. Let $G_{E,p^n} = \Gal(\Q(E[p^n])/\Q)$ and let $G'_{E,p^n}$ denote the commutator subgroup of $G_{E,p^n}$.  
    By Lemma \ref{rootsofunity_in_divfld}, it follows that $\Q(\zeta_{p^n}) \subseteq \Q(E[p^n])$. By Lemma \ref{CMfld_in_divfld}, it follows that $K\subseteq \Q(E[p^n])$. Therefore, $K(\zeta_{p^n})$ is an abelian extension contained in $\Q(E[p^n])$. Thus, an upper bound for the size of the commutator subgroup is,
    \[
    |G'_{E,p^n}| \ \leq \ \frac{|G_{E,p^n}|}{|K(\zeta_{p^n})/\Q|} \ = \ \frac{|G_{E,p^n}|}{2\cdot (p^n - p^{n-1})},
    \]
    where the size of $G_{E,p^n}$ can be derived from Lemma \ref{size_of_cartan}.

    \medskip

    First, let $E/\Q$ be an elliptic curve with CM by an order $\mathcal{O}_{K,f}$, where $j(E)\neq 0$, and let $p \geq 3$. Theorem 1.2.(4) of \cite{lozano-galoiscm} implies that $G_{E,p^n} \cong \mathcal{N}_{\delta,\phi}(p^n)$, and Theorem 1.2.(2) of \cite{lozano-galoiscm} says that $G_{E,p^\infty}$ is defined at level $p$ for all odd primes, if $j(E)\neq 0$.  

    \begin{itemize}
        \item[(i)] If the $G_{E,p^n}$ is the normalizer of a split Cartan subgroup, then we have the following upper bound on the size of the commutator subgroup,
            \[
                |G'_{E,p^n}| \ \leq \ \frac{|G_{E,p^n}|}{|K(\zeta_{p^n})/\Q|} \ = \ \frac{2\cdot p^{2(n-1)}(p-1)^2}{2\cdot p^{n-1}(p-1)} \ = \  p^{n-1}\cdot (p-1).
            \]
        Using induction and the method outlined in Section \ref{main_method}, we will prove that the bound above is also a lower bound on the size of the commutator subgroup, i.e., $|G'_{E,p^n}| \geq p^{n-1}\cdot (p-1)$. 
        
        \underline{Base case:} When $n=1$, one can compute that $|G'_{E,p}| = p-1$. Let $\pi|_{G'} \colon G'_{E,p^2} \to G'_{E,p}$ be the natural reduction map. By Lemma \ref{pi_comm_surj}, $\pi|_{G'}$ is surjective, so it follows that 
        \[
        |G'_{E,p^2}| \ \geq \ |G'_{E,p}|\cdot |\ker(\pi|_{G'})| \ \geq \ (p-1).
        \]
        We can construct a non-trivial commutator element in the kernel of $\pi|_{G'}$ as follows. Let $A_{p^2} = \gamma_1 \bmod p^2$, let $B_{p^2} = c_{\mathrm{sp}}(a,b) \bmod p^2$, and compute $Y_{p^2}$. Take $\kappa = c_{\mathrm{sp}}(1,p+1) \bmod p^2$ and compute $A_{p^2}'$, $B_{p^2}'$, and $Y_{p^2}'$. Using \verb|Magma| to compute the quotient of the commutators $Y_{p^2}$ and $Y_{p^2}'$, we get $Y_{p^2}/Y_{p^2}' = c_{\mathrm{sp}}(1/(p+1),p+1) \bmod p^2 \equiv I \bmod p$ (see Remark \ref{Magma_verified}). Therefore, $|\ker(\pi|_{G'})| \geq p$ and it follows that $|G'_{E,p^2}| \geq (p-1)\cdot p$.
        
        \underline{Induction:} Suppose that $|G'_{E,p^n}| \geq (p-1)\cdot p^{n-1}$ for all $n\geq 1$. Let $\pi|_{G'} \colon G'_{E,p^{n+1}} \to G'_{E,p^n}$ be the natural reduction map. Since $\pi|_{G'}$ is surjective, it follows that $|G'_{E,p^{n+1}}| \geq (p-1)\cdot p^{n-1}$.
        We can construct a non-trivial commutator element in the kernel of $\pi|_{G'}$ as in the base case. Let $A_{p^{n+1}} = \gamma_1 \bmod p^{n+1}$, let $B_{p^{n+1}} = c_{\mathrm{sp}}(a,b) \bmod p^{n+1}$, and compute $Y_{p^{n+1}}$. Take $\kappa = c_{\mathrm{sp}}(1,p^n+1) \bmod p^{n+1}$ and compute $A_{p^{n+1}}'$, $B_{p^{n+1}}'$, and $Y_{p^{n+1}}'$. A \verb|Magma| computation verifies that $Y_{p^{n+1}}/Y_{p^{n+1}}' = c_{\mathrm{sp}}(1/(p^n+1),p^n+1) \bmod p^{n+1} \equiv I \bmod p^n$. Therefore, $|\ker(\pi|_{G'})| \geq p$ and it follows that $|G'_{E,p^{n+1}}| \geq (p-1)\cdot p^{n-1}\cdot p = (p-1)\cdot p^{n}$.

        By induction, we proved that $|G'_{E,p^n}| \geq (p-1)\cdot p^{n-1}$. Therefore, $|G'_{E,p^n}| = (p-1)\cdot p^{n-1}$ for $n\geq 1$, and hence 
        \[
            |M_E(p^n)| \ = \ \frac{|G_{E,p^n}|}{|G'_{E,p^n}|} \ = \ \frac{2\cdot p^{2(n-1)}(p-1)^2}{(p-1)\cdot p^{n-1}} \ = \ 2\cdot p^{n-1} \cdot (p-1),
        \]
        which is the degree of $K(\zeta_{p^n})/\Q$. Thus, $M_E(p^n) = K(\zeta_{p^n})$.

        \item[(ii)] If the $G_{E,p^n}$ is the normalizer of a non-split Cartan subgroup, then we have the following upper bound on the size of the commutator subgroup,
            \[
                |G'_{E,p^n}| \ \leq \ \frac{|G_{E,p^n}|}{|K(\zeta_{p^n})/\Q|} \ = \ \frac{2\cdot p^{2(n-1)}(p^2-1)}{2\cdot p^{n-1}(p-1)} \ = \ p^{n-1}(p+1).
            \]
        Using induction and the method outlined in Section \ref{main_method}, we will prove that the bound above is also a lower bound on the size of the commutator subgroup, i.e., $|G'_{E,p^n}| \geq p^{n-1}\cdot (p+1)$. 
        
        \underline{Base case:} When $n=1$, one can compute that $|G'_{E,p}| = p+1$. Let $\pi|_{G'} \colon G'_{E,p^2} \to G'_{E,p}$ be the natural reduction map. Since $\pi|_{G'}$ is surjective, it follows that 
        \[
        |G'_{E,p^2}| \ \geq \ |G'_{E,p}|\cdot |\ker(\pi|_{G'})| \ \geq \ (p+1).
        \]
        We can construct a non-trivial commutator element in the kernel of $\pi|_{G'}$ as follows. Let $A_{p^2} = c_1 \bmod p^2$, let $B_{p^2} = c_{\mathrm{ns}}(a,b) \bmod p^2$, and compute $Y_{p^2}$. Take $\kappa = c_{\mathrm{ns}}(1,p) \bmod p^2$ and compute $A_{p^2}'$, $B_{p^2}'$, and $Y_{p^2}'$. Using \verb|Magma| to compute the quotient of the commutators $Y_{p^2}$ and $Y_{p^2}'$, we get $Y_{p^2}/Y_{p^2}' = c_{\mathrm{ns}}(1,-2p/(\delta p^2 - 1)) \bmod p^2 \equiv I \bmod p$. Therefore, $|\ker(\pi|_{G'})| \geq p$ and it follows that $|G'_{E,p^2}| \geq (p+1)\cdot p$.
        
        \underline{Induction:} Suppose that $|G'_{E,p^n}| \geq (p+1)\cdot p^{n-1}$ for all $n\geq 1$. Let $\pi|_{G'} \colon G'_{E,p^{n+1}} \to G'_{E,p^n}$ be the natural reduction map. Since $\pi|_{G'}$ is surjective, it follows that $|G'_{E,p^{n+1}}| \geq (p+1)\cdot p^{n-1}$.
        We can construct a non-trivial commutator element in the kernel of $\pi|_{G'}$ as in the base case. Let $A_{p^{n+1}} = \gamma_1 \bmod p^{n+1}$, let $B_{p^{n+1}} = c_{\mathrm{ns}}(a,b) \bmod p^{n+1}$, and compute $Y_{p^{n+1}}$. Take $\kappa = c_{\mathrm{ns}}(1,p^n) \bmod p^{n+1}$ and compute $A_{p^{n+1}}'$, $B_{p^{n+1}}'$, and $Y_{p^{n+1}}'$. A \verb|Magma| computation verifies that $Y_{p^{n+1}}/Y_{p^{n+1}}' = c_{\mathrm{ns}}(1,-2p^n/(\delta p^{2n}-1)) \bmod p^{n+1} \equiv I \bmod p^n$. Therefore, $|\ker(\pi|_{G'})| \geq p$ and it follows that $|G'_{E,p^{n+1}}| \geq (p+1)\cdot p^{n-1}\cdot p = (p+1)\cdot p^{n}$.

        By induction, we proved that $|G'_{E,p^n}| \geq (p+1)\cdot p^{n-1}$. Therefore, $|G'_{E,p^n}| = (p+1)\cdot p^{n-1}$ for $n\geq 1$, and hence
        \[
            |M_E(p^n)| \ = \ \frac{|G_{E,p^n}|}{|G'_{E,p^n}|} \ = \ \frac{2\cdot p^{2(n-1)}(p^2-1)}{(p+1)\cdot p^{n-1}} \ = \ 2\cdot p^{n-1} \cdot (p-1),
        \]
        which is the degree of $K(\zeta_{p^n})/\Q$. Thus, $M_E(p^n) = K(\zeta_{p^n})$.
    \end{itemize}

    Next, let $E/\Q$ be an elliptic curve with $j(E)=0$. Thus, $E$ has CM by $\mathcal{O}_{K,f}$, where $K=\Q(\sqrt{-3})$, $f=1$, and $\mathcal{O}_{K,f} = \mathcal{O}_K = \Z[(1+\sqrt{-3})/2]$. Let $p>3$ be a prime. Then, the $p$-adic image of $\rho_{E,p^\infty}$ is described by Theorem 1.4 of \cite{lozano-galoiscm} and depends on the class of $p \bmod 9$, so we consider three separate cases, which in turn correspond to parts (1), (2), and (3) of [11, Thm. 1.4]. 

    \begin{enumerate}
        \item[(1)] If $p\equiv \pm 1 \bmod 9$, then $G_{E,p^n} \cong \mathcal{N}_{\delta,0}(p^n)$. 
        \begin{itemize}
            \item If $p \equiv 1 \bmod 9$, then $G_{E,p^n}$ is the normalizer of a split Cartan subgroup, and as in part (i) above, we conclude that $M_E(p^n) = K(\zeta_{p^n})$.
            \item If $p \equiv -1 \bmod 9$, then $G_{E,p^n}$ is the normalizer of a non-split Cartan subgroup, and as in part (ii) above, we conclude that $M_E(p^n) = K(\zeta_{p^n})$.
        \end{itemize}

        \item[(2)] If $p\equiv 2$ or $5 \bmod 9$, then $G_{E,p^n}$ is contained in the normalizer of a non-split Cartan subgroup. 
        \begin{itemize}
            \item If $G_{E,p^n} \cong \mathcal{N}_{\delta,0}(p^n)$, then as in part (ii) above, we conclude that $M_E(p^n) = K(\zeta_{p^n})$.

            \item If $G_{E,p^n} \cong \langle \mathcal{C}_{\delta,0}(p^n)^3, c_\varepsilon \rangle$, then $G_{E,p^n}$ is an index $3$ subgroup of the normalizer and we have the following upper bound on the size of the commutator subgroup,
            \[
                |G'_{E,p^n}| \ \leq \ \frac{|G_{E,p^n}|}{|K(\zeta_{p^n})/\Q|} \ = \ \frac{2\cdot p^{2(n-1)}(p^2-1)/3}{2\cdot p^{n-1}(p-1)} \ = \  p^{n-1}\cdot (p+1)/3. 
            \]
            Repeating the induction steps in part (ii) with $B_{p^{n+1}} = c_{\mathrm{ns}}(a,b)^3 \bmod p^{n+1}$ for $n\geq 1$, we prove that $|G'_{E,p^n}| \geq p^{n-1}\cdot (p+1)/3$. 
            Therefore, we conclude that $|G'_{E,p^n}| = p^{n-1}\cdot (p+1)/3$ for $n\geq 1$, and hence 
            \[
                |M_E(p^n)| \ = \ \frac{|G_{E,p^n}|}{|G'_{E,p^n}|} \ = \ \frac{2\cdot p^{2(n-1)}(p^2-1)/3}{(p+1)\cdot p^{n-1}/3} \ = \ 2\cdot p^{n-1} \cdot (p-1),
            \]
            which is the degree of $K(\zeta_{p^n})/\Q$. Thus, $M_E(p^n) = K(\zeta_{p^n})$.
        \end{itemize}

        \item[(3)] If $p \equiv 4$ or $7 \bmod 9$, then $G_{E,p^n}$ is contained in the normalizer of a split Cartan subgroup. 
        \begin{itemize}
            \item If $G_{E,p^n} \cong \mathcal{N}_{\delta,0}(p^n)$, then as in part (i) above, we conclude that $M_E(p^n) = K(\zeta_{p^n})$.

            \item If 
            \[
                G_{E,p^n} \ \cong \ \left\langle \left\{ \begin{pmatrix}a&0\\0&b\end{pmatrix} : \frac{a}{b} \in {(\Z/p^n\Z)^\times}^3 \right\} , \gamma = \begin{pmatrix}0&\varepsilon\\ \varepsilon&0\end{pmatrix}\right\rangle,
            \]
            then $G_{E,p^n}$ is an index $3$ subgroup of the normalizer and we have the following upper bound on the size of the commutator subgroup,
            \[
                |G'_{E,p^n}| \ \leq \ \frac{|G_{E,p^n}|}{|K(\zeta_{p^n})/\Q|} \ = \ \frac{2\cdot p^{2(n-1)}(p-1)^2/3}{2\cdot p^{n-1}(p-1)} \ = \  p^{n-1}\cdot (p-1)/3. 
            \]
            Repeating the induction steps in part (i) with $B_{p^{n+1}} = c_{\mathrm{sp}}(1,0) \bmod p^{n+1}$ for $n\geq 1$, we prove that $|G'_{E,p^n}| \geq p^{n-1}\cdot (p-1)/3$. 
            Therefore, we conclude that $|G'_{E,p^n}| = p^{n-1}\cdot (p-1)/3$ for $n\geq 1$, and hence 
            \[
                |M_E(p^n)| \ = \ \frac{|G_{E,p^n}|}{|G'_{E,p^n}|} \ = \ \frac{2\cdot p^{2(n-1)}(p-1)^2/3}{(p-1)\cdot p^{n-1}/3} \ = \ 2\cdot p^{n-1} \cdot (p-1),
            \]
            which is the degree of $K(\zeta_{p^n})/\Q$. Thus, $M_E(p^n) = K(\zeta_{p^n})$.
        \end{itemize}         
    \end{enumerate}
 
\end{proof}

%%%%%%%%%%%%%%%%%%%%%%%%%%%%%%%%%%%%%%%%%%%%%%%%%%%%%%%%%%%%%%%%%%%%%%%%%%%%%%%%

\subsection{Odd primes dividing the discriminant}\label{sect-pdivdisc} 

Let $p$ be an odd prime dividing $\Delta_Kf$. In this section, we will prove Theorems \ref{pdivDeltaKf} and \ref{3divDeltaKf}. First, we will prove the case where $p \geq 3$ and $j(E)\neq 0$.

\begin{proof}[Proof of Theorem \ref{pdivDeltaKf}]
    Let $E/\Q$ be an elliptic curve with CM by an order $\mathcal{O}_{K,f}$ such that $j_{K,f} \neq 0$ or $1728$. Let $p\geq 3$ be an odd prime dividing $\Delta_Kf$. Let $G_{E,p^n} = \Gal(\Q(E[p^n])/\Q)$ and let $G'_{E,p^n}$ denote the commutator subgroup of $G_{E,p^n}$. 
    By Lemma \ref{rootsofunity_in_divfld}, it follows that $\Q(\zeta_{p^n}) \subseteq \Q(E[p^n])$. By Lemma \ref{CMfld_in_divfld}, it follows that $K\subseteq \Q(E[p^n])$. Note that $K \subseteq \Q(\zeta_{p^n})$. Therefore, $\Q(\zeta_{p^n})$ is an abelian extension contained in $\Q(E[p^n])$. Thus, an upper bound for the size of the commutator subgroup is, 
    \[
        |G'_{E,p^n}| \ \leq \ \frac{|G_{E,p^n}|}{|\Q(\zeta_{p^n})/\Q|} \ = \ \frac{|G_{E,p^n}|}{p^n - p^{n-1}},
    \]
    where the size of $G_{E,p^n}$ can be derived from Lemma \ref{size_of_cartan}. By Theorem 1.5.(a) of \cite{lozano-galoiscm}, if $j_{K,f} \neq 0,1728$, then either $G_{E,p^\infty}\cong \mathcal{N}_{\delta,0}(p^\infty)$ or $G_{E,p^\infty}$ is the following index $2$ subgroup of the normalizer,
    \[
        G_{E,p^\infty} = \left\langle c_\varepsilon, \left\{\begin{pmatrix}a&b\\\delta b&a\end{pmatrix}: a \in \Z_p^{\times^2}, b \in \Z_p \right\} \right\rangle.
    \]

\begin{itemize}
    \item[(1)] Suppose that $G_{E,p^\infty}$ is an index $2$ subgroup of the normalizer. Then we have the following upper bound on the size of the commutator subgroup,
    \[
        |G'_{E,p^n}| \ \leq \ \frac{|G_{E,p^n}|}{|\Q(\zeta_{p^n})/\Q|} \ = \ \frac{2\cdot p^{2n-1}\cdot (p-1)/2}{p^{n-1}\cdot (p-1)} \ = \ p^n.
    \]
    Using induction and the method outlined in Section \ref{main_method}, we will prove that the bound above is also a lower bound on the size of the commutator subgroup, i.e., $|G'_{E,p^n}| \geq p^n$. 

    \underline{Base case:} When $n=1$, one can compute that $|G'_{E,p^n}| = p$. Let $\pi|_{G'} \colon G'_{E,p^2} \to G'_{E,p}$ be the natural reduction map. By Lemma \ref{pi_comm_surj}, $\pi|_{G'}$ is surjective, so it follows that 
        \[
        |G'_{E,p^2}| \ \geq \ |G'_{E,p}|\cdot |\ker(\pi|_{G'})| \ \geq \ p.
        \]
        We can construct a non-trivial commutator element in the kernel of $\pi|_{G'}$ as follows. Let $A_{p^2} = c_\varepsilon \bmod p^2$, let $B_{p^2} = c_{\delta,0}(1,1) \bmod p^2$ so that $B_{p^2} \in G_{E,p^2}$, and compute $Y_{p^2}$. Take $\kappa = c_{\delta,0}(1,p) \bmod p^2$ and compute $A_{p^2}'$, $B_{p^2}'$, and $Y_{p^2}'$. Using \verb|Magma| to compute the quotient of the commutators $Y_{p^2}$ and $Y_{p^2}'$, we get $Y_{p^2}/Y_{p^2}' = c_{\delta,0}(1,2p) \bmod p^2 \equiv I \bmod p$. Therefore, $|\ker(\pi|_{G'})| \geq p$ and it follows that $|G'_{E,p^2}| \geq p\cdot p = p^2$.

        \underline{Induction:} Suppose that $|G'_{E,p^n}| \geq p^n$ for all $n\geq 1$. Let $\pi|_{G'} \colon G'_{E,p^{n+1}} \to G'_{E,p^n}$ be the natural reduction map. Since $\pi|_{G'}$ is surjective, it follows that $|G'_{E,p^{n+1}}| \geq p^n$.
        We can construct a non-trivial commutator element in the kernel of $\pi|_{G'}$ as in the base case. Let $A_{p^{n+1}} = c_\varepsilon \bmod p^{n+1}$, let $B_{p^{n+1}} = c_{\delta,0}(1,1) \bmod p^{n+1}$, and compute $Y_{p^{n+1}}$. Take $\kappa = c_{\delta,0}(1,p^n) \bmod p^{n+1}$ and compute $A_{p^{n+1}}'$, $B_{p^{n+1}}'$, and $Y_{p^{n+1}}'$. A \verb|Magma| computation verifies that $Y_{p^{n+1}}/Y_{p^{n+1}}' = c_{\delta,0}(1,2p^n) \bmod p^{n+1} \equiv I \bmod p^n$. Therefore, $|\ker(\pi|_{G'})| \geq p$ and it follows that $|G'_{E,p^{n+1}}| \geq p^n\cdot p = p^{n+1}$.

        By induction, we proved that $|G'_{E,p^n}| \geq p^n$. Therefore, $|G'_{E,p^n}| = p^n$ for all $n\geq 1$, and hence 
        \[
            |M_E(p^n)| \ = \ \frac{|G_{E,p^n}|}{|G'_{E,p^n}|} \ = \ \frac{2\cdot p^{2n-1}(p-1)/2}{p^n} \ = \ p^{n-1} \cdot (p-1),
        \]
        which is the degree of $\Q(\zeta_{p^n})/\Q$. Thus, $M_E(p^n) = \Q(\zeta_{p^n})$.

        \item[(2)] Suppose that $G_{E,p^\infty} \cong \mathcal{N}_{\delta,0}(p^\infty)$. By Proposition 1.4 of \cite{silverman1}, if two elliptic curves have the same $j$-invariant (not $0,1728$), then they are quadratic twists of each other. Since $j(E)\neq 0$ or $1728$, it follows that the index $1$ images are quadratic twists of the index $2$ images by some $\alpha \in \Z$ that is unique up to squares and described in Definition \ref{defn-alpha}. Therefore, by Lemma \ref{sqrtd_in_divfld}, it follows that $\Q(\sqrt{\alpha}) \subseteq \Q(E[p^n])$. Thus, $\Q(\zeta_{p^n},\sqrt{\alpha})$ is an abelian extension contained in $\Q(E[p^n])$, and we have the following upper bound on the size of the commutator subgroup, 
        \[
            |G'_{E,p^n}| \ \leq \ \frac{|G_{E,p^n}|}{|\Q(\zeta_{p^n},\sqrt{\alpha})/\Q|} \ = \ \frac{2\cdot p^{2n-1}\cdot (p-1)}{2\cdot p^{n-1}\cdot (p-1)} \ = \ p^n.
        \]
        Repeating the induction steps in part (1), we prove that $|G'_{E,p^n}| \geq p^n$.  
        Therefore, we conclude that $|G'_{E,p^n}| = p^n$, and hence 
        \[
            |M_E(p^n)| \ = \ \frac{|G_{E,p^n}|}{|G'_{E,p^n}|} \ = \ \frac{2\cdot p^{2n-1}(p-1)}{p^n} \ = \ 2\cdot p^{n-1} \cdot (p-1),
        \]
        which is the degree of $\Q(\zeta_{p^n},\sqrt{\alpha})/\Q$. Thus, $M_E(p^n) = \Q(\zeta_{p^n},\sqrt{\alpha})$.
\end{itemize}

\end{proof}

Next, we will prove the case where $p=3$ and $j(E) = 0$. 

\begin{proof}[Proof of Theorem \ref{3divDeltaKf}]
    Let $E/\Q$ be an elliptic curve with CM by an order $\mathcal{O}_{K,f}$. Let $j_{K,f} = 0$, so $\Delta_Kf^2 = -3$, and let $\delta = -3/4$ and $\phi = 0$. 
    Let $\alpha = \alpha(E)$ be as in Definition \ref{defn-alpha}.
    Let $G_{E,3^n} = \Gal(\Q(E[3^n])/\Q)$ and let $G'_{E,3^n}$ denote the commutator subgroup of $G_{E,3^n}$. 
    By Lemma \ref{rootsofunity_in_divfld}, it follows that $\Q(\zeta_{3^n}) \subseteq \Q(E[3^n])$. Note that $K = \Q(\sqrt{-3}) \subseteq \Q(\zeta_{3^n})$. Therefore, $\Q(\zeta_{3^n})$ is an abelian extension contained in $\Q(E[3^n])$. Thus, an upper bound for the size of the commutator subgroup is 
    \[
        |G'_{E,3^n}| \ \leq \ \frac{|G_{E,3^n}|}{|\Q(\zeta_{3^n})/\Q|} \ = \ \frac{|G_{E,3^n}|}{3^n - 3^{n-1}},
    \]
    where the size of $G_{E,3^n}$ can be derived from Lemma \ref{size_of_cartan}. By Theorem 1.5.(b) of \cite{lozano-galoiscm}, either $G_{E,3^\infty} \cong \mathcal{N}_{\delta,0}(3^\infty)$ or $[\mathcal{N}_{\delta,0}(3^\infty) : G_{E,3^\infty}]=2,3,$ or $6$, and the level of definition of the image is at most $27$.

\begin{itemize}
    \item[(1)] Suppose that $[\mathcal{N}_{\delta,0}(3^\infty) : G_{E,3^\infty}] = 2$, where $G_{E,3^\infty}$ is the following index $2$ subgroup, 
    \[
        G_{E,3^\infty} = \left\langle c_\varepsilon, \left\{\begin{pmatrix}a&b\\-3b/4&a\end{pmatrix}: a,b \in \Z_3, a \equiv 1 \bmod 3 \right\} \right\rangle \subseteq \GL(2,\Z_3).
    \]
    Note that by Lemma \ref{alpha-for-p3-j0}, we have that $\sqrt{\alpha} \in \Q(\zeta_{3^n})$. Therefore, $\Q(\zeta_{3^n})$ remains the largest abelian extension that we know is contained in $\Q(E[3^n])$, and 
    we have the following upper bound on the size of the commutator subgroup, 
    \[
        |G'_{E,3^n}| \ \leq \ \frac{|G_{E,3^n}|}{|\Q(\zeta_{3^n})/\Q|} \ = \ \frac{2\cdot 3^{2n-1}\cdot (3-1)/2}{3^{n-1}\cdot (3-1)} \ = \ 3^n.
    \]
    Using induction and the method outlined in Section \ref{main_method}, we will prove that the bound above is also a lower bound on the size of the commutator subgroup, i.e., $|G'_{E,3^n}| \geq 3^n$. 
    
    \underline{Base case:} Since the level of definition of $G_{E,3^n}$ is at most $27$, we will compute $|G'_{E,3^n}|$ for the first three values of $n$. A \verb|Magma| computation shows that $|G'_{E,3}| = 3$, $|G'_{E,9}| = 9$, and $|G'_{E,27}| = 27$.
    Let $\pi|_{G'}:G'_{E,81} \to G'_{E,27}$ be the natural reduction map. By Lemma \ref{pi_comm_surj}, $\pi|_{G'}$ is surjective, so it follows that 
        \[
        |G'_{E,81}| \ \geq \ |G'_{E,27}|\cdot |\ker(\pi|_{G'})| \ \geq \ 27.
        \]
    We can construct a non-trivial commutator element in the kernel of $\pi|_{G'}$ as follows. Let $A_{81} = c_\varepsilon \bmod 81$, let $B_{81} = c_{\delta,0}(1,1) \bmod 81$ so that $B_{81} \in G_{E,81}$, and compute $Y_{81}$. Take $\kappa = c_{\delta,0}(1,27) \bmod 81$ and compute $A_{81}'$, $B_{81}'$, and $Y_{81}'$. Using \verb|Magma| to compute the quotient of the commutators $Y_{81}$ and $Y_{81}'$, we get $Y_{81}/Y_{81}' = c_{\delta,0}(1,54a^2/(a^2 + 3b^2/4)) \bmod 81 \equiv I \bmod 27$. Therefore, $|\ker(\pi|_{G'})| \geq 3$ and it follows that $|G'_{E,81}| \geq 27\cdot 3 = 81$.

    \underline{Induction:} Suppose that $|G'_{E,3^n}| \geq 3^n$ for all $n\geq 1$. Let $\pi|_{G'} \colon G'_{E,3^{n+1}} \to G'_{E,3^n}$ be the natural reduction map. Since $\pi|_{G'}$ is surjective, it follows that $|G'_{E,3^{n+1}}| \geq 3^n$.
    We can construct a non-trivial commutator element in the kernel of $\pi|_{G'}$ as in the base case. Let $A_{3^{n+1}} = c_\varepsilon \bmod 3^{n+1}$, let $B_{3^{n+1}} = c_{\delta,0}(1,1) \bmod 3^{n+1}$, and compute $Y_{3^{n+1}}$. Take $\kappa = c_{\delta,0}(1,3^n) \bmod 3^{n+1}$ and compute $A_{3^{n+1}}'$, $B_{3^{n+1}}'$, and $Y_{3^{n+1}}'$. A \verb|Magma| computation verifies that $Y_{3^{n+1}}/Y_{3^{n+1}}' = c_{\delta,0}(1,2\cdot 3^na^2/(a^2 + 3b^2/4)) \bmod 3^{n+1} \equiv I \bmod 3^n$. Therefore, $|\ker(\pi|_{G'})| \geq 3$ and it follows that $|G'_{E,3^{n+1}}| \geq 3^n\cdot 3 = 3^{n+1}$.

    By induction, we proved that $|G'_{E,3^n}| \geq 3^n$. Therefore, $|G'_{E,3^n}| = 3^n$ for all $n\geq 1$, and hence
        \[
            |M_E(3^n)| \ = \ \frac{|G_{E,3^n}|}{|G'_{E,3^n}|} \ = \ \frac{2\cdot 3^{2n-1}(3-1)/2}{3^n} \ = \ 3^{n-1} \cdot (3-1),
        \]
    which is the degree of $\Q(\zeta_{3^n})/\Q$. Thus, $M_E(3^n) = \Q(\zeta_{3^n})$.

    \item[(2)] Suppose that $G_{E,3^\infty} \cong \mathcal{N}_{\delta,0}(3^\infty)$. By Lemma \ref{alpha-for-p3-j0}, if $[\mathcal{N}_{\delta,0}(3^n):G_{E,3^n}]=1$, then $\Q(\sqrt{\alpha}) \subseteq \Q(E[3^n])$. Thus, $\Q(\zeta_{3^n},\sqrt{\alpha})$ is an abelian extension contained in $\Q(E[3^n])$, and we have the following upper bound on the size of the commutator subgroup, 
        \[
            |G'_{E,3^n}| \ \leq \ \frac{|G_{E,3^n}|}{|\Q(\zeta_{3^n},\sqrt{\alpha})/\Q|} \ = \ \frac{2\cdot 3^{2n-1}\cdot (3-1)}{2\cdot 3^{n-1}\cdot (3-1)} \ = \ 3^{n}.
        \]
    Repeating the induction steps in part (1), we prove that $|G'_{E,3^n}| \geq 3^{n}$.  
    Therefore, we conclude that $|G'_{E,3^n}| = 3^{n}$, and hence 
        \[
            |M_E(3^n)| \ = \ \frac{|G_{E,3^n}|}{|G'_{E,3^n}|} \ = \ \frac{2\cdot 3^{2n-1}(3-1)}{3^n} \ = \ 2\cdot 3^{n-1} \cdot (3-1),
        \]
    which is the degree of $\Q(\zeta_{3^n},\sqrt{\alpha})/\Q$. Thus, $M_E(3^n) = \Q(\zeta_{3^n},\sqrt{\alpha})$.

    \item[(3)] Suppose that $[\mathcal{N}_{\delta,0}(3^\infty) : G_{E,3^\infty}] = 6$, where $G_{E,3^\infty}$ is one of the following index $6$ subgroups, 
        \begin{align*}
             G_{E,3^\infty}^1 = \left\langle c_\varepsilon, \left\{\begin{pmatrix}a&b\\-3b/4&a\end{pmatrix}: a \equiv 1, b \equiv 0 \bmod 3\Z_3 \right\} \right\rangle, 
        \end{align*}
        \begin{align*}
             G_{E,3^\infty}^2 = \left\langle c_\varepsilon, \begin{pmatrix}4&0\\0&4\end{pmatrix}, \begin{pmatrix}1&1\\-3/4&1\end{pmatrix}\right\rangle, \quad \text{or} \quad G_{E,3^\infty}^3 = \left\langle c_\varepsilon, \begin{pmatrix}4&0\\0&4\end{pmatrix}, \begin{pmatrix}-5/4&1/2\\-3/8&-5/4\end{pmatrix}\right\rangle.
        \end{align*}
    Note that by Lemma \ref{alpha-for-p3-j0}, we have that $\sqrt{\alpha} \in \Q(\zeta_{3^n})$. Therefore, $\Q(\zeta_{3^n})$ remains the largest abelian extension that we know is contained in $\Q(E[3^n])$, and we have the following upper bound on the size of the commutator subgroup, 
        \[
            |G'_{E,3^n}| \ \leq \ \frac{|G_{E,3^n}|}{|\Q(\zeta_{3^n})/\Q|} \ = \ \frac{2\cdot 3^{2n-1}\cdot (3-1)/6}{3^{n-1}\cdot (3-1)} \ = \ 3^{n-1}.
        \]
    Since the level of definition of $G_{E,3^n}$ is at most $27$, we will compute $|G'_{E,3^n}|$ for the first three values of $n$. A \verb|Magma| computation shows that $|(G^1_{E,3})'| = 1$ and $|(G^2_{E,3})'| = |(G^3_{E,3})'| =3$, and for all images $|G'_{E,9}| = 9$ and $|G'_{E,27}| = 27$.
    Repeating the induction steps in part (1) using $B_{3^{n+1}} = c_{\delta,0}(1,0) \bmod 3^{n+1}$ for $G_{E,3^{n+1}}^1$, $B_{3^{n+1}} = c_{\delta,0}(1,1) \bmod 3^{n+1}$ for $G_{E,3^{n+1}}^2$, and $B_{3^{n+1}} = c_{\delta,0}(-5/4,1/2) \bmod 3^{n+1}$ for $G_{E,3^{n+1}}^3$ for $n\geq 2$, we prove that $|G'_{E,3^n}| \geq 3^{n-1}$. 
    Therefore, we conclude that $|G'_{E,3^n}| = 3^{n-1}$ for $n\geq 2$, and hence 
        \[
            |M_E(3^n)| \ = \ \frac{|G_{E,3^n}|}{|G'_{E,3^n}|} \ = \ \frac{2\cdot 3^{2n-1}(3-1)/6}{3^{n-1}} \ = \ 3^{n-1} \cdot (3-1),
        \]
    which is the degree of $\Q(\zeta_{3^n})/\Q$. Thus, $M_E(3^n) = \Q(\zeta_{3^n})$.

    \item[(4)] Suppose that $[\mathcal{N}_{\delta,0}(p^n):G_{E,3^n}]=3$, where $G_{E,3^\infty}$ is one of the following index $3$ subgroups, 
        \begin{align*}
             G_{E,3^\infty}^1 = \left\langle c_\varepsilon, \left\{\begin{pmatrix}a&b\\-3b/4&a\end{pmatrix}: a \in \Z_3^\times, b \equiv 0 \bmod 3 \right\} \right\rangle, 
        \end{align*}
        \begin{align*}
             G_{E,3^\infty}^2 = \left\langle c_\varepsilon, \begin{pmatrix}2&0\\0&2\end{pmatrix}, \begin{pmatrix}1&1\\-3/4&1\end{pmatrix}\right\rangle, \quad \text{or} \quad G_{E,3^\infty}^3 = \left\langle c_\varepsilon, \begin{pmatrix}2&0\\0&2\end{pmatrix}, \begin{pmatrix}-5/4&1/2\\-3/8&-5/4\end{pmatrix}\right\rangle.
        \end{align*}
        By Lemma \ref{alpha-for-p3-j0}, if $[\mathcal{N}_{\delta,0}(p^n):G_{E,3^n}]=3$, then $\Q(\sqrt{\alpha}) \subseteq \Q(E[3^n])$. Thus, $\Q(\zeta_{3^n},\sqrt{\alpha})$ is an abelian extension contained in $\Q(E[3^n])$, and we have the following upper bound on the size of the commutator subgroup, 
        \[
            |G'_{E,3^n}| \ \leq \ \frac{|G_{E,3^n}|}{|\Q(\zeta_{3^n},\sqrt{\alpha})/\Q|} \ = \ \frac{2\cdot 3^{2n-1}\cdot (3-1)/3}{2\cdot 3^{n-1}\cdot (3-1)} \ = \ 3^{n-1}.
        \]
    Since the level of definition of $G_{E,3^n}$ is at most $27$, we will compute $|G'_{E,3^n}|$ for the first three values of $n$. A \verb|Magma| computation shows that $|(G^1_{E,3})'| = 1$ and $|(G^2_{E,3})'| = |(G^3_{E,3})'| =3$, and for all images $|G'_{E,9}| = 9$ and $|G'_{E,27}| = 27$.
    Repeating the induction steps in part (3), we prove that $|G'_{E,3^n}| \geq 3^{n-1}$. 
    Therefore, we conclude that $|G'_{E,3^n}| = 3^{n-1}$ for $n\geq 2$, and hence 
        \[
            |M_E(3^n)| \ = \ \frac{|G_{E,3^n}|}{|G'_{E,3^n}|} \ = \ \frac{2\cdot 3^{2n-1}(3-1)/3}{3^{n-1}} \ = \ 2\cdot 3^{n-1} \cdot (3-1),
        \]
    which is the degree of $\Q(\zeta_{3^n},\sqrt{\alpha})/\Q$. Thus, $M_E(3^n) = \Q(\zeta_{3^n},\sqrt{\alpha})$.
    
\end{itemize}

\end{proof}

%%%%%%%%%%%%%%%%%%%%%%%%%%%%%%%%%%%%%%%%%%%%%%%%%%%%%%%%%%%%%%%%%%%%%%%%%%%%%%%%

\subsection{$p=2$ cases}\label{sect-casep2}

In this section, we will prove Theorem \ref{casep2}. We begin by proving a lemma that will be used in the following proposition. 

\begin{lemma}\label{conjugates}
    Let $E/\Q$ be an elliptic curve with CM by $\mathcal{O}_{K,f}$, $f\geq 1$, such that $\Delta_Kf^2 \neq -3$. Suppose $2$ is inert in $K$ and does not divide $\Delta_Kf^2$. Then the images $G_{E,2^\infty}$ are all conjugate to each other in $\GL(2,\Z_2)$. 
\end{lemma}

\begin{proof} 
    Let $E/\Q$ be an elliptic curve with CM by $\mathcal{O}_{K,f}$, $f\geq 1$, such that $\Delta_Kf^2 \neq 3$. 
    Since $2$ is inert in $K$ and does not divide $\Delta_Kf$, it follows that $\Delta_Kf^2 \in \{-11, -19, -27, -43, -67,-163\}$ and $G_{E,2^\infty} \cong \mathcal{N}_{\delta,\phi}(2^\infty)$. 
    By Proposition 12.1.4 of \cite{elladic}, it follows that the level of definition of all $2$-adic images is $16$.
    Thus, it suffices to check that for each pair of $\Delta_Kf^2$ in the list, the corresponding $\mathcal{N}_{\delta,\phi}(2^n)$ are both conjugate to each other in $\GL(2,\Z/16\Z)$. A finite computation in \cite{GitHubPaperCode} confirms that the $\mathcal{N}_{\delta,\phi}(2^n)$ are all conjugate to each other in $\GL(2,\Z/16\Z)$, and thus in $\GL(2,\Z/2^n\Z)$ since $16$ is the level of definition of $\mathcal{N}_{\delta,\phi}(2^n)$.
\end{proof}

Suppose $2$ does not divide $\Delta_Kf$. We begin by considering the case where $j(E) \neq 0$.

\begin{proposition}\label{2-ndiv-disc}
    Let $E/\Q$ be an elliptic curve with CM by an order in an imaginary field $K$. Suppose $2$ does not divide $\Delta_Kf$ and that $\Delta_Kf^2 \neq -3$. Then $M_E(2^n) = K(\zeta_{2^n})$ for $n\geq 1$.
\end{proposition}

\begin{proof}
    Let $E/\Q$ be an elliptic curve with CM by an order $\mathcal{O}_{K,f}$, $f\geq 1$, such that $\Delta_Kf^2 \neq -3$. Suppose that $2$ does not divide $\Delta_Kf$. Let $G_{E,2^n} = \Gal(\Q(E[2^n])/\Q)$ and let $G'_{E,2^n}$ denote the commutator subgroup of $G_{E,2^n}$. By Lemma \ref{rootsofunity_in_divfld}, it follows that $\Q(\zeta_{2^n}) \subseteq \Q(E[2^n])$. 
    By Lemma \ref{CMfld_in_divfld}, it follows that $K\subseteq \Q(E[2^n])$ for $n\geq 2$. Therefore, $K(\zeta_{2^n})$ is an abelian extension contained in $\Q(E[2^n])$. Thus, an upper bound for the size of the commutator subgroup is, 
    \[
    |G'_{E,2^n}| \ \leq \ \frac{|G_{E,2^n}|}{|K(\zeta_{2^n})/\Q|} \ = \ \frac{|G_{E,2^n}|}{2\cdot (2^n - 2^{n-1})}, 
    \]
    where the size of $G_{E,2^n}$ can be derived from Lemma \ref{size_of_cartan}. By Theorem 1.6 of \cite{lozano-galoiscm}, since $j_{K,f} \neq 0$, it follows that $G_{E,2^n} \cong \mathcal{N}_{\delta,\phi}(2^n)$, and the level of definition of the image is at most $16$. 

    \begin{itemize}
        \item[(1)] If $G_{E,2^n}$ is the normalizer of a split Cartan subgroup, then we have the following upper bound on the size of the commutator subgroup,
        \[
            |G'_{E,2^n}| \ \leq \ \frac{|G_{E,2^n}|}{|K(\zeta_{2^n})/\Q|} \ = \ \frac{2\cdot 2^{2(n-1)}(2-1)^2}{2^n} \ = \ 2^{n-1}. 
        \]
        Using induction and the method outlined in Section \ref{main_method}, we will prove that the bound above is also a lower bound on the size of the commutator subgroup, i.e., $|G'_{E,2^n}| \geq 2^{n-1}$. Note that the only discriminant that falls under this category is $\Delta_Kf^2 = -7$. We will prove the result for $\Delta_Kf^2 = -7$, i.e., $\delta = -2$ and $\phi = 1$. 

        \underline{Base case:} Since the level of definition of $G_{E,2^n}$ is at most $16$, we will compute $|G'_{E,2^n}|$ for the first four values of $n$. 
        A \verb|Magma| computation shows that $|G'_{E,2}| = 1$, $|G'_{E,4}| = 2$, $|G'_{E,8}| = 4$, and $|G'_{E,16}| = 8$. 
        Let $\pi|_{G'} \colon G'_{E,32} \to G'_{E,16}$ be the natural reduction map. By Lemma \ref{pi_comm_surj}, $\pi|_{G'}$ is surjective, so it follows that 
        \[
            |G'_{E,32}| \ \geq \ |G'_{E,16}|\cdot |\ker(\pi|_{G'})| \ \geq \ 8.
        \]
        We can construct a non-trivial commutator element in the kernel of $\pi|_{G'}$ as follows. Let $A_{32} = c_\varepsilon \bmod 32$, let $B_{32} = c_{-2,1}(1,1) \bmod 32$, and compute $Y_{32}$. Take $\kappa = c_{-2,1}(17,16) \bmod 32$ and compute $A_{32}'$, $B_{32}'$, and $Y_{32}'$. Using \verb|Magma| to compute the quotient of the commutators $Y_{32}$ and $Y_{32}'$, we get $Y_{32}/Y_{32}' = c_{\delta,1}(17,16) \bmod 32 \equiv I \bmod 16$. Therefore, $|\ker(\pi|_{G'})| \geq 2$ and it follows that $|G'_{E,32}| \geq 8\cdot 2 = 16$.

        \underline{Induction:} Suppose that $|G'_{E,2^n}| \geq 2^{n-1}$ for all $n\geq 1$. Let $\pi|_{G'} \colon G'_{E,2^{n+1}} \to G'_{E,2^n}$ be the natural reduction map. Since $\pi|_{G'}$ is surjective, it follows that $|G'_{E,2^{n+1}}| \geq 2^{n-1}$. We can construct a non-trivial commutator element in the kernel of $\pi|_{G'}$ as in the base case. Let $A_{2^{n+1}} = c_\varepsilon \bmod 2^{n+1}$, let $B_{2^{n+1}} = c_{-2,1}(1,1) \bmod 2^{n+1}$, and compute $Y_{2^{n+1}}$. Take $\kappa = c_{-1,1}(1+2^n,2^n) \bmod 2^{n+1}$ and compute $A_{2^{n+1}}'$, $B_{2^{n+1}}'$, and $Y_{2^{n+1}}'$. A \verb|Magma| computation verifies that $Y_{2^{n+1}}/Y_{2^{n+1}}' = c_{-1,1}(1 + 2^n, 2^n/(1+2^n)) \bmod 2^{n+1} \equiv I \bmod 2^n$. Therefore, $|\ker(\pi|_{G'})| \geq 2$ and it follows that $|G'_{E,2^{n+1}}| \geq 2^{n-1}\cdot 2 = 2^{n}$.

        By induction, we proved that $|G'_{E,2^n}| \geq 2^{n-1}$. Therefore, $|G'_{E,2^n}| = 2^{n-1}$ for $n\geq 1$, and hence 
        \[
            |M_E(2^n)| \ = \ \frac{|G_{E,2^n}|}{|G'_{E,2^n}|} \ = \ \frac{2\cdot 2^{2(n-1)}(2-1)^2}{2^{n-1}} \ = \ 2^{n},
        \]
        which is the degree of $K(\zeta_{2^n})/\Q$. Thus, $M_E(2^n) = K(\zeta_{2^n})$.

        \item[(2)] If $G_{E,2^n}$ is the normalizer of a non-split Cartan subgroup, then we have the following upper bound on the size of the commutator subgroup, 
        \[
            |G'_{E,2^n}| \ \leq \ \frac{|G_{E,2^n}|}{|K(\zeta_{2^n})/\Q|} \ = \ \frac{2\cdot 2^{2(n-1)}(2^2-1)}{2^n} \ = \ 2^{n-1}\cdot 3. 
        \]
        Using induction and the method outlined in Section \ref{main_method}, we will prove that the bound above is also a lower bound on the size of the commutator subgroup, i.e., $|G'_{E,2^n}| \geq 2^{n-1}\cdot 3$. First, note that the discriminants that fall under this category are $\Delta_Kf^2 = -11, -19, -27, -43, -67,$ and $-163$. By Lemma \ref{conjugates}, all $2$-adic images for these discriminants are conjugates of each other in $\GL(2,\Z_2)$. 
        Therefore, without loss of generality, it suffices to prove the result for a choice of $\Delta_Kf^2$. 
        We prove the result for $\Delta_Kf^2 = -11$, i.e., $\delta = -3$ and $\phi = 1$.

        \underline{Base case:} Since the level of definition of $G_{E,2^n}$ is at most $16$, we will compute $|G'_{E,2^n}|$ for the first four values of $n$. A \verb|Magma| computation shows that $|G'_{E,2}| = 3$, $|G'_{E,4}| = 6$, $|G'_{E,8}| = 12$, and $|G'_{E,16}| = 24$. 
        Let $\pi|_{G'} \colon G'_{E,32} \to G'_{E,16}$ be the natural reduction map. Since $\pi|_{G'}$ is surjective, it follows that 
        \[
            |G'_{E,32}| \ \geq \ |G'_{E,16}|\cdot |\ker(\pi|_{G'})| \ \geq \ 24.
        \]
        We can construct a non-trivial commutator element in the kernel of $\pi|_{G'}$ as follows. Let $A_{32} = c_\varepsilon \bmod 32$, let $B_{32} = c_{-3,1}(1,0) \bmod 32$, and compute $Y_{32}$. Take $\kappa = c_{-3,1}(17,16) \bmod 32$ and compute $A_{32}'$, $B_{32}'$, and $Y_{32}'$. Using \verb|Magma| to compute the quotient of the commutators $Y_{32}$ and $Y_{32}'$, we get $Y_{32}/Y_{32}' = c_{\delta,1}(17,0) \bmod 32 \equiv I \bmod 16$. Therefore, $|\ker(\pi|_{G'})| \geq 2$ and it follows that $|G'_{E,32}| \geq 12\cdot 2 = 24$.

        \underline{Induction:} Suppose that $|G'_{E,2^n}| \geq 2^{n-1}\cdot 3$ for all $n\geq 1$. Let $\pi|_{G'} \colon G'_{E,2^{n+1}} \to G'_{E,2^n}$ be the natural reduction map. Since $\pi|_{G'}$ is surjective, it follows that $|G'_{E,2^{n+1}}| \geq 2^{n-1}\cdot 3$. We can construct a non-trivial commutator element in the kernel of $\pi|_{G'}$ as in the base case. Let $A_{2^{n+1}} = c_\varepsilon \bmod 2^{n+1}$, let $B_{2^{n+1}} = c_{-3,1}(1,0) \bmod 2^{n+1}$, and compute $Y_{2^{n+1}}$. Take $\kappa = c_{-3,1}(1+2^n,2^n) \bmod 2^{n+1}$ and compute $A_{2^{n+1}}'$, $B_{2^{n+1}}'$, and $Y_{2^{n+1}}'$. A \verb|Magma| computation verifies that $Y_{2^{n+1}}/Y_{2^{n+1}}' = c_{-3,1}(1 + 2^n, 0) \bmod 2^{n+1} \equiv I \bmod 2^n$. Therefore, $|\ker(\pi|_{G'})| \geq 2$ and it follows that $|G'_{E,2^{n+1}}| \geq 2^{n-1}\cdot 3 \cdot 2 = 2^{n}\cdot 3$.

        By induction, we proved that $|G'_{E,2^n}| \geq 2^{n-1}\cdot 3$. Therefore, $|G'_{E,2^n}| = 2^{n-1}\cdot 3$ for $n\geq 1$, and hence
        \[
            |M_E(2^n)| \ = \ \frac{|G_{E,2^n}|}{|G'_{E,2^n}|} \ = \ \frac{2\cdot 2^{2(n-1)}(2^2-1)}{2^{n-1}\cdot 3} \ = \ 2^n,
        \]
        which is the degree of $K(\zeta_{2^n})/\Q$. Thus, $M_E(2^n) = K(\zeta_{2^n})$.
    \end{itemize}

\end{proof}

Next, we consider the case where $2$ does not divide $\Delta_Kf$ and $j(E) = 0$.

\begin{proposition}\label{2-ndiv-disc-j0}
    Let $E/\Q$ be an elliptic curve with $j(E) = 0$. Then $M_E(2^n) = K(\zeta_{2^n})$ for $n\geq 1$.
\end{proposition}

\begin{proof}
    Let $E/\Q$ be an elliptic curve with $j(E)=0$. Thus, $E$ has CM by $\mathcal{O}_{K,f}$, where $K=\Q(\sqrt{-3})$. Since $\Delta_K = -3$ and $f=1$, we set $\delta = -1$ and $\phi = 1$. Let $G_{E,2^n} = \Gal(\Q(E[2^n])/\Q)$ and let $G'_{E,2^n}$ denote the commutator subgroup of $G_{E,2^n}$. 
    By Lemma \ref{rootsofunity_in_divfld}, it follows that $\Q(\zeta_{2^n}) \subseteq \Q(E[2^n])$. By Lemma \ref{CMfld_in_divfld}, it follows that $K \subseteq \Q(E[2^n])$. Therefore, $K(\zeta_{2^n})$ is an abelian extension contained in $\Q(E[2^n])$. Thus, an upper bound for the size of the commutator subgroup is, 
    \[
    |G'_{E,2^n}| \ \leq \ \frac{|G_{E,2^n}|}{|K(\zeta_{2^n})/\Q|} \ = \ \frac{|G_{E,2^n}|}{2\cdot (2^n - 2^{n-1})},
    \]
    where the size of $G_{E,2^n}$ can be derived from Lemma \ref{size_of_cartan}. By Theorem 1.8 of \cite{lozano-galoiscm}, if $j(E)=0$, then either $G_{E,2^\infty} \cong \mathcal{N}_{-1,1}(2^\infty)$ or $G_{E,2^\infty}$ is the following index $3$ subgroup, 
    \[
        G_{E,2^\infty} = \left\langle \gamma', -\operatorname{Id}, \begin{pmatrix}7&4\\-4&3\end{pmatrix},\begin{pmatrix}3&6\\-6&-3\end{pmatrix}\right\rangle, 
    \]
    where $\gamma' \in \left\{\begin{pmatrix}0&1\\1&0\end{pmatrix}, \begin{pmatrix}0&-1\\-1&0\end{pmatrix} \right\}$, and $\gamma \equiv \gamma' \bmod 4$.

\begin{itemize}
    \item[(1)] Suppose that $G_{E,2^\infty}$ is an index $3$ subgroup of the normalizer. Then we have the following upper bound on the size of the commutator subgroup, 
    \[
        |G'_{E,2^n}| \ \leq \ \frac{|G_{E,2^n}|}{|K(\zeta_{2^n})/\Q|} \ = \ \frac{2\cdot 2^{2(n-1)}\cdot (2^2-1)/3}{2^{n}} \ = \ 2^{n-1}.
    \]
    Using induction and the method outlined in Section \ref{main_method}, we will prove that the bound above is also a lower bound on the size of the commutator subgroup, i.e., $|G'_{E,2^n}| \geq 2^{n-1}$. 

    \underline{Base case:} Since the level of definition of $G_{E,2^n}$ is at most $16$, we will compute $|G'_{E,2^n}|$ for the first four values of $n$. 
    A \verb|Magma| computation shows that $|G'_{E,2}| = 1$, $|G'_{E,4}| = 2$, $|G'_{E,8}| = 4$, and $|G'_{E,16}| = 8$.
    Let $\pi|_{G'} \colon G'_{E,32} \to G'_{E,16}$ be the natural reduction map. By Lemma \ref{pi_comm_surj}, $\pi|_{G'}$ is surjective, so it follows that 
    \[
        |G'_{E,32}| \ \geq \ |G'_{E,16}|\cdot |\ker(\pi|_{G'})| \ \geq \ 8.
        \]
    We can construct a non-trivial commutator element in the kernel of $\pi|_{G'}$ as follows. Let $A_{32} = \gamma' \bmod 32$, let $B_{32} = c_{-1,1}(-3,6) \bmod 32$, and compute $Y_{32}$. Take $\kappa = c_{-1,1}(17,16) \bmod 32$ and compute $A_{32}'$, $B_{32}'$, and $Y_{32}'$. Using \verb|Magma| to compute the quotient of the commutators $Y_{32}$ and $Y_{32}'$, we get $Y_{32}/Y_{32}' = c_{-1,1}(17,0) \bmod 32 \equiv I \bmod 16$. Therefore, $|\ker(\pi|_{G'})| \geq 2$ and it follows that $|G'_{E,32}| \geq 8\cdot 2 = 16$.

    \underline{Induction:} Suppose that $|G'_{E,2^n}| \geq 2^{n-1}$ for all $n\geq 1$. Let $\pi|_{G'} \colon G'_{E,2^{n+1}} \to G'_{E,2^n}$ be the natural reduction map. Since $\pi|_{G'}$ is surjective, it follows that $|G'_{E,2^{n+1}}| \geq 2^{n-1}$.
    We can construct a non-trivial commutator element in the kernel of $\pi|_{G'}$ as in the base case. Let $A_{2^{n+1}} = \gamma' \bmod 2^{n+1}$, let $B_{2^{n+1}} = c_{-1,1}(-3,6) \bmod 2^{n+1}$, and compute $Y_{2^{n+1}}$. Take $\kappa = c_{-1,1}(1+2^n,2^n) \bmod 2^{n+1}$ and compute $A_{2^{n+1}}'$, $B_{2^{n+1}}'$, and $Y_{2^{n+1}}'$. A \verb|Magma| computation verifies that $Y_{2^{n+1}}/Y_{2^{n+1}}' = c_{-1,1}(1 + 2^n, 0) \bmod 2^{n+1} \equiv I \bmod 2^n$. Therefore, $|\ker(\pi|_{G'})| \geq 2$ and it follows that $|G'_{E,2^{n+1}}| \geq 2^{n-1}\cdot 2 = 2^{n}$.

    By induction, we proved that $|G'_{E,2^n}| \geq 2^{n-1}$. Therefore, $|G'_{E,2^n}| = 2^{n-1}$ for $n\geq 1$, and hence
        \[
            |M_E(2^n)| \ = \ \frac{|G_{E,2^n}|}{|G'_{E,2^n}|} \ = \ \frac{2\cdot 2^{2(n-1)}(2^2-1)/3}{2^{n-1}} \ = \ 2^{n-1},
        \]
    which is the degree of $K(\zeta_{2^n})/\Q$. Thus, $M_E(2^n) = K(\zeta_{2^n})$.

    \item[(2)] Suppose that $G_{E,2^\infty} \cong \mathcal{N}_{-1,1}(2^\infty)$. Then we have the following upper bound on the size of the commutator subgroup, 
    \[
        |G'_{E,2^n}| \ \leq \ \frac{|G_{E,2^n}|}{|K(\zeta_{2^n})/\Q|} \ = \ \frac{2\cdot 2^{2(n-1)}\cdot (2^2-1)}{2^{n}} \ = \ 2^{n-1}\cdot 3.
    \]
    Since the level of definition of $G_{E,2^n}$ is at most $16$, we will compute $|G'_{E,2^n}|$ for the first four values of $n$.
    A \verb|Magma| computation shows that $|G'_{E,2}| = 3$, $|G'_{E,4}| = 6$, $|G'_{E,8}| = 12$, and $|G'_{E,16}| = 24$.
    Repeating the induction steps in part (1) using the appropriate choice of $B_{2^{n+1}}^1 = c_{-1,1}(1,1) \bmod 2^{n+1}$ for all $n\geq 4$, we prove that $|G'_{E,2^n}| \geq 2^{n-1}\cdot 3$. 
    Therefore, we conclude that $|G'_{E,2^n}| = 2^{n-1}\cdot 3$ for $n\geq 1$, and hence
        \[
            |M_E(2^n)| \ = \ \frac{|G_{E,2^n}|}{|G'_{E,2^n}|} \ = \ \frac{2\cdot 2^{2(n-1)}(2^2-1)}{2^{n-1}\cdot 3} \ = \ 2^{n-1},
        \]
    which is the degree of $K(\zeta_{2^n})/\Q$. Thus, $M_E(2^n) = K(\zeta_{2^n})$.
\end{itemize}

\end{proof}

\noindent The proofs of Propositions \ref{2-ndiv-disc} and \ref{2-ndiv-disc-j0} conclude the proof of Theorem \ref{casep2}.(a).

Suppose $2$ divides $\Delta_Kf$. First we will consider the cases where $\Delta_Kf^2 = -12$ or $-28$.

\begin{proof}[Proof of Theorem \ref{casep2}.(b).(i)]
    Let $E/\Q$ be an elliptic curve with CM by $\mathcal{O}_{K,f}$ such that $j_{K,f}\neq 0$ or $1728$. Suppose $\Delta_Kf^2 \equiv 0 \bmod 2$, but $\Delta_Kf^2 \not\equiv 0 \bmod 8$ (i.e., $\Delta_Kf^2 = -12$ or $-28$). 
    Let $\alpha = \alpha(E)$ be as in Definition \ref{defn-alpha}.
    Let $G_{E,2^n} = \Gal(\Q(E[2^n])/\Q)$ and let $G'_{E,2^n}$ denote the commutator subgroup of $G_{E,2^n}$. By \cite[Table 1]{gonzalez-jimenez-lozano-robledo}, if $\Delta_Kf^2 = -12$, then $\Q(E[2]) = \Q(\sqrt{3})$; and if $\Delta_Kf^2 = -28$, then $\Q(E[2]) = \Q(\sqrt{7})$. Therefore, throughout this proof we will focus on $M_E(2^n)$ for $n\geq 2$. 
    
    By Lemma \ref{moreroots_in_divfld}, since $\Delta_Kf^2 \equiv 0 \bmod 2$, it follows that $\Q(\zeta_{2^{n+1}}) \subseteq \Q(E[2^n])$. By Lemma \ref{CMfld_in_divfld}, it follows that $K\subseteq \Q(E[2^n])$. Therefore, $K(\zeta_{2^{n+1}})$ is an abelian extension contained in $\Q(E[2^n])$. By Theorem 1.6 of \cite{lozano-galoiscm}, if $\Delta_Kf^2 \not\equiv 0 \bmod 8$ and $j_{K,f}\neq 0,1728$, then $G_{E,2^\infty} \cong \mathcal{N}_{\delta,\phi}.(2^\infty)$. Thus, an upper bound for the size of the commutator subgroup is, 
    \[
    |G'_{E,2^n}| \ \leq \ \frac{|G_{E,2^n}|}{|K(\zeta_{2^{n+1}})/\Q|} \ = \ \frac{2\cdot 2^{2n-1}(2-1)}{2\cdot (2^{n+1}-2^n)} \ = \ 2^{n-1}. 
    \]
    Using induction and the method outlined in Section \ref{main_method}, we will prove that the bound above is also a lower bound on the size of the commutator subgroup, i.e., $|G'_{E,2^n}| \geq 2^{n-1}$. 

    \underline{Base case:} Since the level of definition of $G_{E,2^n}$ is at most $16$, we will compute $|G'_{E,2^n}|$ for the first four values of $n$. 
    A \verb|Magma| computation shows that $|G'_{E,2}| = 1$, $|G'_{E,4}| = 2$, $|G'_{E,8}| = 4$, and $|G'_{E,16}| = 8$. Let $\pi|_{G'}\colon G'_{E,32} \to G'_{E,16}$ be the natural reduction map. By Lemma \ref{defn-alpha}, $\pi|_{G'}$ is surjective, so it follows that
    \[
        |G'_{E,32}| \ \geq \ |G'_{E,16}|\cdot |\ker(\pi|_{G'})| \ \geq \ 8.
    \]
    We can construct a non-trivial commutator element in the kernel of $\pi|_{G'}$ as follows. Let $A_{32} = c_\varepsilon \bmod 32$, let $B_{32} = c_{\delta,0}(a,b) \bmod 32$, and compute $Y_{32}$. Take $\kappa = c_{\delta,0}(1,8) \bmod 32$ and compute $A_{32}'$, $B_{32}'$, and $Y_{32}'$. Using \verb|Magma| to compute the quotient of the commutators $Y_{32}$ and $Y_{32}'$, we get $Y_{32}/Y_{32}' = c_{\delta,0}(1,16) \bmod 32 \equiv I \bmod 16$. Therefore, $|\ker(\pi|_{G'})| \geq 2$ and it follows that $|G'_{E,32}| \geq 8\cdot 2 = 16$.

    \underline{Induction:} Suppose that $|G'_{E,2^n}| \geq 2^{n-1}$ for all $n\geq 1$. Let $\pi|_{G'} \colon G'_{E,2^{n+1}} \to G'_{E,2^n}$ be the natural reduction map. Since $\pi|_{G'}$ is surjective, it follows that $|G'_{E,2^{n+1}}| \geq 2^{n-1}$. We can construct a non-trivial commutator element in the kernel of $\pi|_{G'}$ as in the base case. Let $A_{2^{n+1}} = c_\varepsilon \bmod 2^{n+1}$, let $B_{2^{n+1}} = c_{\delta,0}(a,b) \bmod 2^{n+1}$, and compute $Y_{2^{n+1}}$. Take $\kappa = c_{\delta,0}(1,2^{n-1}) \bmod 2^{n+1}$ and compute $A_{2^{n+1}}'$, $B_{2^{n+1}}'$, and $Y_{2^{n+1}}'$. A \verb|Magma| computation verifies that $Y_{2^{n+1}}/Y_{2^{n+1}}' = c_{\delta,0}(1,2^n) \bmod 2^{n+1} \equiv I \bmod 2^n$. Therefore, $|\ker(\pi|_{G'})| \geq 2$ and it follows that $|G'_{E,2^{n+1}}| \geq 2^{n-1}\cdot 2 = 2^{n}$.

        By induction, we proved that $|G'_{E,2^n}| \geq 2^{n-1}$. Therefore, $|G'_{E,2^n}| = 2^{n-1}$ for $n\geq 1$, and hence 
        \[
            |M_E(2^n)| \ = \ \frac{|G_{E,2^n}|}{|G'_{E,2^n}|} \ = \ \frac{2\cdot 2^{2n-1}(2-1)}{2^{n-1}} \ = \ 2^{n+1},
        \]
        which is the degree of $K(\zeta_{2^{n+1}})/\Q$. Thus, $M_E(2^n) = K(\zeta_{2^{n+1}})$ for $n\geq 2$.
\end{proof}

Next, we handle the case where $\Delta_Kf^2 = -4$. 

\begin{proof}[Proof of Theorem \ref{casep2}.(b).(ii)]
    Let $E/\Q$ be an elliptic curve with $j(E)=1728$. Thus, $E$ has CM by $\mathcal{O}_{K,f} = \Z[i]$, where $K=\Q(i)$. Since $\Delta_K = -4$ and $f=1$, we set $\delta = -1$ and $\phi = 0$. 
    Let $\alpha = \alpha(E)$ be as in Definition \ref{defn-alpha}.
    Let $G_{E,2^n} = \Gal(\Q(E[2^n])/\Q)$ and let $G'_{E,2^n}$ denote the commutator subgroup of $G_{E,2^n}$. By \cite[Table 1]{gonzalez-jimenez-lozano-robledo}, it follows that
    \begin{itemize}
        \item if $E : y^2 = x^3 + t^2x$ for $t \in \Q^*$, then $M_E(2) = \Q(i)$ and $M_E(4) = \Q(\zeta_8,\sqrt{t})$;
        \item if $E : y^2 = x^3 - t^2x$ for $t \in \Q^*$, then $M_E(2) = \Q$ and $M_E(4) = \Q(\zeta_8,\sqrt{t})$;  
        \item if $E : y^2 = x^3 + sx$ where $s \neq \pm t^2$ for $t \in \Q^*$, then $M_E(2) = \Q(\sqrt{-s})$ and $M_E(4) = \Q(\zeta_8)$.
    \end{itemize}
    Therefore, throughout this proof we will focus on $M_E(2^n)$ for $n \geq 3$.
    
    By Lemma \ref{moreroots_in_divfld}, since $\Delta_Kf^2 \equiv 0 \bmod 2$, it follows that $\Q(\zeta_{2^{n+1}}) \subseteq \Q(E[2^n])$. Note that $K \subseteq \Q(\zeta_{2^{n+1}})$. Therefore, $\Q(\zeta_{2^{n+1}})$ is an abelian extension contained in $\Q(E[2^n])$. Thus, an upper bound for the size of the commutator subgroup is, 
    \[
    |G'_{E,2^n}| \ \leq \ \frac{|G_{E,2^n}|}{|\Q(\zeta_{2^{n+1}})/\Q|} \ = \ \frac{|G_{E,2^n}|}{2^{n+1} - 2^n},
    \]
    where the size of $G_{E,2^n}$ can be derived from Lemma \ref{size_of_cartan}. Let $c\in \Gal(\overline{\Q}/\Q)$ be a complex conjugation, and $\gamma = \rho_{E,2^\infty}(c)$. By Theorem 1.7 of \cite{lozano-galoiscm}, we know that $[\mathcal{N}_{-1,0}(2^\infty):G_{E,2^\infty}]=1,2,$ or $4$ and that $G_{E,2^\infty} = \langle \gamma', G_{E,K,2^\infty} \rangle$ where 
        \[
            \gamma' \in \left\{c_1 = \begin{pmatrix}1&0\\0&-1\end{pmatrix}, c_{-1} = \begin{pmatrix}-1&0\\0&1\end{pmatrix}, c_1' = \begin{pmatrix}0&1\\1&0\end{pmatrix}, c_{-1}' = \begin{pmatrix}0&-1\\-1&0\end{pmatrix} \right\},
        \]
    such that $\gamma \equiv \gamma' \bmod 4$, and $G_{E,K,2^\infty} = \rho_{E,2^\infty}(G_{\Q(i)})$. 

    \begin{itemize}
        \item[(1)] Suppose that $[\mathcal{N}_{-1,0}(2^\infty) : G_{E,2^\infty}] = 4$, where $G_{E,2^\infty}$ is generated by $\gamma'$ and one of the following groups
        \begin{align*}
            G_{4,a} = \left\langle 5\cdot \operatorname{Id}, \begin{pmatrix}1&2\\-2&1\end{pmatrix} \right\rangle, \quad G_{4,b} = \left\langle 5\cdot \operatorname{Id}, \begin{pmatrix}-1&-2\\2&-1\end{pmatrix} \right\rangle, \\
            G_{4,c} = \left\langle -3\cdot \operatorname{Id}, \begin{pmatrix}2&-1\\1&2\end{pmatrix} \right\rangle, \quad G_{4,d} = \left\langle -3\cdot \operatorname{Id}, \begin{pmatrix}-2&1\\-1&-2\end{pmatrix} \right\rangle.
        \end{align*}
        Note that by Lemma \ref{alpha-for-p2-j1728}, we have that $\sqrt{\alpha} \in \Q(\zeta_{2^{n+1}})$. Therefore, $\Q(\zeta_{2^{n+1}})$ remains the largest abelian extension that we know is contained in $\Q(E[2^n])$, and we have the following upper bound on the size of the commutator subgroup, 
        \[
            |G'_{E,2^n}| \ \leq \ \frac{|G_{E,2^n}|}{|\Q(\zeta_{2^{n+1}})/\Q|} \ = \ \frac{2\cdot 2^{2n-1}(2-1)/4}{2^n\cdot (2-1)} \ = \ 2^{n-2}. 
        \]
        Using induction and the method outlined in Section \ref{main_method}, we will prove that the bound above is also a lower bound on the size of the commutator subgroup, i.e., $|G'_{E,2^n}| \geq 2^{n-2}$. 

        \underline{Base case:} Since the level of definition of $G_{E,2^n}$ is at most $16$, we will compute $|G'_{E,2^n}|$ for the first four values of $n$. 
        A \verb|Magma| computation shows that $|G'_{E,2}| = 1$, for $G_{4,a}$ and $G_{4,b}$ we have $|G'_{E,4}| = 1$, for $G_{4,c}$ and $G_{4,d}$ we have $|G'_{E,4}| = 2$, and for all images $|G'_{E,8}| = 2$ and $|G'_{E,16}| = 4$.  
        Let $\pi|_{G'}\colon G'_{E,32} \to G'_{E,16}$ be the natural reduction map. By Lemma \ref{pi_comm_surj}, $\pi|_{G'}$ is surjective, so it follows that
        \[
        |G'_{E,32}| \ \geq \ |G'_{E,16}|\cdot |\ker(\pi|_{G'})| \ \geq \ 4.
        \]
        We can construct a non-trivial commutator element in the kernel of $\pi|_{G'}$ as follows. Let $A_{32} = \gamma' \bmod 32$, let $B_{32} = c_{-1,0}(a,b) \bmod 32$ such that $c_{-1,0}(a,b) \in G_{E,2^n}$, and compute $Y_{32}$. Take $\kappa = c_{-1,0}(1,8) \bmod 32$ and compute $A_{32}'$, $B_{32}'$, and $Y_{32}'$. Using \verb|Magma| to compute the quotient of the commutators $Y_{32}$ and $Y_{32}'$, we get $Y_{32}/Y_{32}' = c_{-1,0}(1,16) \bmod 32 \equiv I \bmod 16$. Therefore, $|\ker(\pi|_{G'})| \geq 2$ and it follows that $|G'_{E,32}| \geq 4\cdot 2 = 8$.

        \underline{Induction:} Suppose that $|G'_{E,2^{n}}| \geq 2^{n-2}$ for all $n\geq 1$. Let $\pi|_{G'} \colon G'_{E,2^{n+1}} \to G'_{E,2^{n}}$ be the natural reduction map. Since $\pi|_{G'}$ is surjective, it follows that $|G'_{E,2^{n+1}}| \geq 2^{n-2}$.
        We can construct a non-trivial commutator element in the kernel of $\pi|_{G'}$ as in the base case. Let $A_{2^{n+1}} = \gamma' \bmod 2^{n+1}$, let $B_{2^{n+1}} = c_{-1,0}(a,b) \bmod 2^{n+1}$ such that $c_{-1,0}(a,b) \in G_{E,2^{n+1}}$, and compute $Y_{2^{n+1}}$. Take $\kappa = c_{-1,0}(1,2^{n-1}) \bmod 2^{n+1}$ and compute $A_{2^{n+1}}'$, $B_{2^{n+1}}'$, and $Y_{2^{n+1}}'$. A \verb|Magma| computation verifies that $Y_{2^{n+1}}/Y_{2^{n+1}}' = c_{-1,0}(1, 2^{n}) \bmod 2^{n+1} \equiv I \bmod 2^{n}$. Therefore, $|\ker(\pi|_{G'})| \geq 2$ and it follows that $|G'_{E,2^{n+1}}| \geq 2^{n-2}\cdot 2 = 2^{n-1}$.

        By induction, we proved that $|G'_{E,2^n}| \geq 2^{n-2}$. Therefore, $|G'_{E,2^n}| = 2^{n-2}$ for $n\geq 3$, and hence
        \[
            |M_E(2^n)| \ = \ \frac{|G_{E,2^n}|}{|G'_{E,2^n}|} \ = \ \frac{2\cdot 2^{2n-1}(2-1)/4}{2^{n-2}} \ = \ 2^{n},
        \]
        which is the degree of $\Q(\zeta_{2^{n+1}})/\Q$. Thus, $M_E(2^n) = \Q(\zeta_{2^{n+1}})$ for $n \geq 3$.

        \item[(2)] Suppose that $[\mathcal{N}_{-1,0}(2^\infty) : G_{E,2^\infty}] = 2$, where $G_{E,2^\infty}$ is generated by $\gamma'$ and one of the following groups
        \begin{align*}
            G_{2,a} = \left\langle -\operatorname{Id}, 3\cdot \operatorname{Id}, \begin{pmatrix}1&2\\-2&1\end{pmatrix} \right\rangle \quad \text{or} \quad G_{2,b} = \left\langle -\operatorname{Id}, 3\cdot \operatorname{Id}, \begin{pmatrix}2&1\\-1&2\end{pmatrix} \right\rangle.
        \end{align*}
        Then by Lemma \ref{alpha-for-p2-j1728}, we have that $\Q(\sqrt{\alpha}) \subseteq \Q(E[2^n])$ for $n\geq 2$.
        Thus, $\Q(\zeta_{2^{n+1}},\sqrt{\alpha})$ is an abelian extension contained in $\Q(E[2^n])$, 
        and we have the following upper bound on the size of the commutator subgroup, 
        \[
            |G'_{E,2^n}| \ \leq \ \frac{|G_{E,2^n}|}{|\Q(\zeta_{2^{n+1}},\sqrt{\alpha})/\Q|} \ = \ \frac{2\cdot 2^{2n-1}(2-1)/2}{2\cdot 2^n} \ = \ 2^{n-2}. 
        \]
        Since the level of definition of $G_{E,2^n}$ is at most $16$, we will compute $|G'_{E,2^n}|$ for the first four values of $n$.
        A \verb|Magma| computation shows that $|G'_{E,2}| = 1$, for $G_{2,a}$ we have $|G'_{E,4}| = 1$, for $G_{2,b}$ we have $|G'_{E,4}| = 2$ , and for all images $|G'_{E,8}| = 4$ and $|G'_{E,16}| = 8$.
        Repeating the induction steps in part (1), we prove that $|G'_{E,2^n}| \geq 2^{n-2}$.
        Therefore, we conclude that $|G'_{E,2^n}| = 2^{n-2}$ for $n\geq 3$, and hence
        \[
            |M_E(2^n)| \ = \ \frac{|G_{E,2^n}|}{|G'_{E,2^n}|} \ = \ \frac{2\cdot 2^{2n-1}(2-1)/2}{2^{n-2}} \ = \ 2^{n+1},
        \]
        which is the degree of $\Q(\zeta_{2^{n+1}},\sqrt{\alpha})/\Q$. Thus, $M_E(2^n) = \Q(\zeta_{2^{n+1}},\sqrt{\alpha})$ for $n \geq 3$. 

        \item[(3)] Suppose that $G_{E,2^\infty} \cong \mathcal{N}_{-1,0}(2^\infty)$. Then by Lemma \ref{alpha-for-p2-j1728}, we have that $\Q(\sqrt{\alpha})\subseteq \Q(E[2^n])$. Thus, $\Q(\zeta_{2^{n+1}},\sqrt{\alpha})$ is an abelian extension contained in $\Q(E[2^n])$, and we have the following upper bound on the size of the commutator subgroup, 
        \[
            |G'_{E,2^n}| \ \leq \ \frac{|G_{E,2^n}|}{|\Q(\zeta_{2^{n+1}},\sqrt{\alpha})/\Q|} \ = \ \frac{2\cdot 2^{2n-1}(2-1)}{2\cdot 2^n} \ = \ 2^{n-1}. 
        \]
        Since the level of definition of $G_{E,2^n}$ is at most $16$, we will compute $|G'_{E,2^n}|$ for the first four values of $n$.
        A \verb|Magma| computation shows that $|G'_{E,2}| = 1$, $|G'_{E,4}| = 2$, $|G'_{E,8}| = 4$, and $|G'_{E,16}| = 8$.
        Repeating the induction steps in part (1), we prove that $|G'_{E,2^n}| \geq 2^{n-1}$. 
        Therefore, we conclude that $|G'_{E,2^n}| = 2^{n-1}$ for $n\geq 2$, and hence 
        \[
            |M_E(2^n)| \ = \ \frac{|G_{E,2^n}|}{|G'_{E,2^n}|} \ = \ \frac{2\cdot 2^{2n-1}(2-1)}{2^{n-1}} \ = \ 2^{n+1},
        \]
        which is the degree of $\Q(\zeta_{2^{n+1}},\sqrt{\alpha})/\Q$. Thus, $M_E(2^n) = \Q(\zeta_{2^{n+1}},\sqrt{\alpha})$ for $n\geq 2$. 
    \end{itemize}
\end{proof}

Lastly, we handle the case where $\Delta_Kf^2 = -8$ or $-16$. 

\begin{proof}[Proof of Theorem \ref{casep2}.(b).(iii)]
    Let $E/\Q$ be an elliptic curve with CM by $\mathcal{O}_{K,f}$ such that $\Delta_Kf^2 \equiv 0 \bmod 8$ (i.e., $\Delta_Kf^2 = -8$ or $-16$). Since $\Delta_Kf^2 \equiv 0 \bmod 4$, we set $\delta = \Delta_Kf^2/4$ and $\phi = 0 $. 
    Let $\alpha = \alpha(E)$ be as in Definition \ref{defn-alpha}.
    Let $G_{E,2^n} = \Gal(\Q(E[2^n])/\Q)$ and let $G'_{E,2^n}$ denote the commutator subgroup of $G_{E,2^n}$. By Lemma \ref{alpha-for-p2-8-16}, we have that $\Q(E[2])=\Q(\sqrt{2})$ which is an abelian extension of $\Q$, so throughout this proof we will focus on $M_E(2^n)$ for $n\geq 2$. 
    
    By Lemma \ref{moreroots_in_divfld}, since $\Delta_Kf^2 \equiv 0 \bmod 2$, it follows that $\Q(\zeta_{2^{n+1}}) \subseteq \Q(E[2^n])$ for $n\geq 2$. Note that $K \subseteq \Q(\zeta_{2^{n+1}})$. Therefore, $\Q(\zeta_{2^{n+1}})$ is an abelian extension contained in $\Q(E[2^n])$. Thus, an upper bound for the size of the commutator subgroup is, 
    \[
    |G'_{E,2^n}| \ \leq \ \frac{|G_{E,2^n}|}{|\Q(\zeta_{2^{n+1}})/\Q|} \ = \ \frac{|G_{E,2^n}|}{2^{n+1}-2^n},
    \]
    where the size of $G_{E,2^n}$ can be derived from Lemma \ref{size_of_cartan}. By Theorem 9.3 of \cite{lozano-galoiscm}, either $G_{E,2^\infty} \cong \mathcal{N}_{\delta,0}(2^\infty)$ or $G_{E,2^\infty}$ is one of the following index $2$ subgroups of the normalizer, where $\beta \in \{3,5\}$,
    \[
        G_{E,2^\infty}^1 = \left\langle c_\varepsilon, \begin{pmatrix}\beta&0\\0&\beta\end{pmatrix}, \begin{pmatrix}1&1\\\delta&1\end{pmatrix} \right\rangle \quad \text{or} \quad G_{E,2^\infty}^2 = \left\langle c_\varepsilon, \begin{pmatrix}\beta&0\\0&\beta\end{pmatrix}, \begin{pmatrix}-1&-1\\-\delta&-1\end{pmatrix} \right\rangle.
    \]

    \begin{itemize}
        \item[(1)] Suppose that $[\mathcal{N}_{\delta,0}(2^\infty) : G_{E,2^\infty}] = 2$. Note that by Lemma \ref{alpha-for-p2-8-16}, we have $\sqrt{\alpha} \in \Q(\zeta_{2^{n+1}})$. Therefore, $\Q(\zeta_{2^{n+1}})$ remains the largest abelian extension that we know is contained in $\Q(E[2^n])$, and we have the following upper bound on the size of the commutator subgroup, 
        \[
            |G'_{E,2^n}| \ \leq \ \frac{|G_{E,2^n}|}{|\Q(\zeta_{2^{n+1}})/\Q|} \ = \ \frac{2\cdot 2^{2n-1}(2-1)/2}{2^n} \ = \ 2^{n-1}.
        \]
        Using induction and the method outlined in Section \ref{main_method}, we will prove that the bound above is also a lower bound on the size of the commutator subgroup, i.e., $|G'_{E,2^n}| \geq 2^{n-1}$.

        \underline{Base case:} Since the level of definition of $G_{E,2^n}$ is at most $16$, we will compute $|G'_{E,2^n}|$ for the first four values of $n$. 
        A \verb|Magma| computation shows that $|G'_{E,2}| = 1$, $|G'_{E,4}| = 2$, $|G'_{E,8}| = 4$, and $|G'_{E,16}| = 8$.
        Let $\pi|_{G'}\colon G'_{E,32} \to G'_{E,16}$ be the natural reduction map. By Lemma \ref{pi_comm_surj}, $\pi|_{G'}$ is surjective, so it follows that
        \[
        |G'_{E,32}| \ \geq \ |G'_{E,16}|\cdot |\ker(\pi|_{G'})| \ \geq \ 8.
        \]
        We can construct a non-trivial commutator element in the kernel of $\pi|_{G'}$ as follows. Let $A_{32} = c_\varepsilon \bmod 32$, let $B_{32} = c_{\delta,0}(1,1) \bmod 32$ (note that we will get the same results if we use $B_{32} = c_{\delta,0}(-1,-1) \bmod 32$), and compute $Y_{32}$. Take $\kappa = c_{\delta,0}(1,8) \bmod 32$ and compute $A_{32}'$, $B_{32}'$, and $Y_{32}'$. Using \verb|Magma| to compute the quotient of the commutators $Y_{32}$ and $Y_{32}'$, we get $Y_{32}/Y_{32}' = c_{\delta,0}(1,16) \bmod 32 \equiv I \bmod 16$. Therefore, $|\ker(\pi|_{G'})| \geq 2$ and it follows that $|G'_{E,32}| \geq 8\cdot 2 = 16$.

        \underline{Induction:} Suppose that $|G'_{E,2^{n}}| \geq 2^{n-1}$ for all $n\geq 1$. Let $\pi|_{G'} \colon G'_{E,2^{n+1}} \to G'_{E,2^{n}}$ be the natural reduction map. Since $\pi|_{G'}$ is surjective, it follows that $|G'_{E,2^{n+1}}| \geq 2^{n-2}$.
        We can construct a non-trivial commutator element in the kernel of $\pi|_{G'}$ as in the base case. Let $A_{2^{n+1}} = c_\varepsilon \bmod 2^{n+1}$, let $B_{2^{n+1}} = c_{\delta,0}(1,1) \bmod 2^{n+1}$, and compute $Y_{2^{n+1}}$. Take $\kappa = c_{\delta,0}(1,2^{n-1}) \bmod 2^{n+1}$ and compute $A_{2^{n+1}}'$, $B_{2^{n+1}}'$, and $Y_{2^{n+1}}'$. A \verb|Magma| computation verifies that $Y_{2^{n+1}}/Y_{2^{n+1}}' = c_{\delta,0}(1, 2^{n}) \bmod 2^{n+1} \equiv I \bmod 2^{n}$. Therefore, $|\ker(\pi|_{G'})| \geq 2$ and it follows that $|G'_{E,2^{n+1}}| \geq 2^{n-1}\cdot 2 = 2^{n}$.

        By induction, we proved that $|G'_{E,2^n}| \geq 2^{n-1}$. Therefore, $|G'_{E,2^n}| = 2^{n-1}$ for $n\geq 2$, and hence
        \[
            |M_E(2^n)| \ = \ \frac{|G_{E,2^n}|}{|G'_{E,2^n}|} \ = \ \frac{2\cdot 2^{2n-1}(2-1)/2}{2^{n-1}} \ = \ 2^{n},
        \]
        which is the degree of $\Q(\zeta_{2^{n+1}})/\Q$. Thus, $M_E(2^n) = \Q(\zeta_{2^{n+1}})$ for $n\geq 2$.

        \item[(2)] Suppose that $G_{E,2^\infty} \cong \mathcal{N}_{\delta,0}(2^\infty)$. Then by Lemma \ref{alpha-for-p2-8-16}, we have that $\Q(E[2])=\Q(\sqrt{2})$, and if $[\mathcal{N}_{\delta,0}(2^n) : G_{E,2^n}] = 1$, then $\Q(\sqrt{\alpha}) \subseteq \Q(E[2^n])$ for $n\geq 2$. Thus, $\Q(\zeta_{2^{n+1}},\sqrt{\alpha})$ is an abelian extension contained in $\Q(E[2^n])$, and we have the following upper bound on the size of the commutator subgroup, 
        \[
            |G'_{E,2^n}| \ \leq \ \frac{|G_{E,2^n}|}{|\Q(\zeta_{2^{n+1}},\sqrt{\alpha})/\Q|} \ = \ \frac{2\cdot 2^{2n-1}(2-1)}{2\cdot 2^n} \ = \ 2^{n-1}.
        \] 
        Since the level of definition of $G_{E,2^n}$ is at most 16, we will compute $|G'_{E,2^n}|$ for the first four values of $n$. A \verb|Magma| computation shows that $|G'_{E,2}| = 1$, $|G'_{E,4}| = 2$, $|G'_{E,8}| = 4$, and $|G'_{E,16}| = 8$.
        Repeating the induction steps in part (1), we prove that $|G'_{E,2^n}| \geq 2^{n-1}$.
        Therefore, we conclude that $|G'_{E,2^n}| = 2^{n-1}$ for $n\geq 2$, and hence
        \[
            |M_E(2^n)| \ = \ \frac{|G_{E,2^n}|}{|G'_{E,2^n}|} \ = \ \frac{2\cdot 2^{2n-1}(2-1)}{2^{n-1}} \ = \ 2^{n+1},
        \]
        which is the degree of $\Q(\zeta_{2^{n+1}},\sqrt{\alpha})/\Q$. Thus, $M_E(2^n) = \Q(\zeta_{2^{n+1}},\sqrt{\alpha})/\Q$ for $n\geq 2$.
    \end{itemize}
\end{proof}

%%%%%%%%%%%%%%%%%%%%%%%%%%%%%%%%%%%%%%%%%%%%%%%%%%%%%%%%%%

\bibliography{bibliography}
\bibliographystyle{plain}

%%%%%%%%%%%%%%%%%%%%%%%%%%%%%%%%%%%%%%%%%%%%%%%%%%%%%%%%%%

\end{document}